\documentclass[onefignum,onetabnum]{siamart220329}

\usepackage{lipsum}
\usepackage{amsfonts}
\usepackage{graphicx}
\usepackage{epstopdf}
\usepackage{algorithmic}
\usepackage{enumerate}
\usepackage{mathtools}
\usepackage{amssymb}
\usepackage{mathrsfs}
\usepackage{multirow}
\ifpdf
  \DeclareGraphicsExtensions{.eps,.pdf,.png,.jpg}
\else
  \DeclareGraphicsExtensions{.eps}
\fi

\usepackage{enumitem}
\setlist[enumerate]{leftmargin=.5in}
\setlist[itemize]{leftmargin=.5in}

\newsiamremark{remark}{Remark}
\newsiamremark{hypothesis}{Hypothesis}
\crefname{hypothesis}{Hypothesis}{Hypotheses}
\newsiamthm{claim}{Claim}

\newtheorem{example}[theorem]{Example}

\newtheorem{thm}[theorem]{Theorem}

\headers{Sparse Polynomial Optimization with Unbounded Sets}{L. Huang, S. Kang, J. Wang, and H. Yang}

\title{Sparse Polynomial Optimization with Unbounded Sets\thanks{Submitted to the editors DATE.
\funding{This work was supported by the NSFC under grants 12201618 and 12171324.}}}

\author{Lei Huang\thanks{Department of Mathematics, University of California San Diego
  (\email{leh010@ucsd.edu}).} \and Shucheng Kang\thanks{School of Engineering and Applied Sciences, Harvard University (\email{skang1@g.harvard.edu}).} \and Jie Wang\thanks{Academy of Mathematics and Systems Science, Chinese Academy of Sciences (AMSS-CAS)
  (\email{wangjie212@amss.ac.cn}).} \and Heng Yang\thanks{School of Engineering and Applied Sciences, Harvard University (\email{hankyang@seas.harvard.edu}).}}

\usepackage{amsopn}

\newcommand{\R}{{\mathbb R}}

\newcommand{\N}{\mathbb{N}}
\newcommand{\bx}{\mathbf{x}}
\newcommand{\by}{\mathbf{y}}

\newcommand{\w}{\mathbf{w}}
\newcommand{\bu}{\mathbf{u}}

\def\ba{{\boldsymbol{\alpha}}}

\newcommand{\bv}{\mathbf{v}}

\newcommand{\re}{\mathbb{R}}

\newcommand{\cl}[1]{\mathrm{cl}(#1)}

\newcommand{\mR}{\mathbb{R}}

\newcommand{\nn}{\nonumber}

\newcommand{\reff}[1]{(\ref{#1})}

\newcommand{\qmod}[1]{\mbox{QM}[#1]}
\newcommand{\ideal}[1]{\mbox{Ideal}[#1]}
\newcommand{\st}{\mbox{\rm{s.t.}}}
\newcommand{\bdes}{\begin{description}}
\newcommand{\edes}{\end{description}}
\newcommand{\bal}{\begin{align}}
\newcommand{\eal}{\end{align}}
\newcommand{\bnum}{\begin{enumerate}}
\newcommand{\enum}{\end{enumerate}}
\newcommand{\bit}{\begin{itemize}}
\newcommand{\eit}{\end{itemize}}
\newcommand{\bea}{\begin{eqnarray}}
\newcommand{\eea}{\end{eqnarray}}
\newcommand{\be}{\begin{equation}}
\newcommand{\ee}{\end{equation}}
\newcommand{\baray}{\begin{array}}
\newcommand{\earay}{\end{array}}
\newcommand{\bsry}{\begin{subarray}}
\newcommand{\esry}{\end{subarray}}
\newcommand{\bca}{\begin{cases}}
\newcommand{\eca}{\end{cases}}
\newcommand{\bcen}{\begin{center}}
\newcommand{\ecen}{\end{center}}
\newcommand{\bbm}{\begin{bmatrix}}
\newcommand{\ebm}{\end{bmatrix}}
\newcommand{\bmx}{\begin{matrix}}
\newcommand{\emx}{\end{matrix}}
\newcommand{\bpm}{\begin{pmatrix}}
\newcommand{\epm}{\end{pmatrix}}
\newcommand{\btab}{\begin{tabular}}
\newcommand{\etab}{\end{tabular}}

\newif\ifcomment
\commentfalse
\commenttrue

\begin{document}

\maketitle

\begin{abstract}
This paper considers sparse polynomial optimization with unbounded sets. When the problem possesses correlative sparsity, we propose a sparse homogenized Moment-SOS hierarchy with perturbations to solve it. 
The new hierarchy  introduces one extra auxiliary variable for each variable clique according to the correlative sparsity pattern. Under the running intersection property, we prove that this hierarchy has asymptotic convergence. Furthermore, we provide two alternative sparse hierarchies to remove perturbations while preserving asymptotic convergence. As byproducts, new Positivstellens\"atze are obtained for sparse positive polynomials on unbounded sets.
Extensive numerical experiments demonstrate the power of our approach in solving sparse polynomial optimization problems on unbounded sets with up to thousands of variables. Finally, we apply our approach to tackle two trajectory optimization problems (block-moving with minimum work and optimal control of Van der Pol).
\end{abstract}



\section{Introduction}
In this paper, we consider the polynomial optimization problem (POP): 
\be\label{1.1}
\left\{\baray{lll}
\inf & f(\bx) \\
\st  & g_{j}(\bx) \geq0,\quad j =1,\dots,m,
\earay \right.
\ee
where $f(\bx), g_j(\bx)$ are polynomials in $\bx \coloneqq (x_1,\dots,x_n) \in \re^n$.
Let $K$ denote the feasible set of $(\ref{1.1})$ and let $f_{\min}$ denote the optimal value of (\ref{1.1}). Throughout the paper, we assume that $f_{\min}>-\infty$.	
The Moment-SOS hierarchy proposed by Lasserre~\cite{Las01} is efficient in solving \reff{1.1}. Under the Archimedeanness of constraining polynomials (the feasible set $K$ must be compact in this case; see \cite{Las01,laurent2007semidefinite}), it yields a sequence of semidefinite relaxations whose optimal values converge to $f_{\min}$. Furthermore, it was shown in \cite{huang2023finite,nie2014optimality} that the Moment-SOS hierarchy converges in finitely many steps if  standard optimality conditions hold at every global minimizer. We refer  to  books and surveys \cite{lasserre2015introduction,laurent2009sums,marshall2008positive,nie2023moment} for more general introductions to polynomial optimization.

When the feasible set $K$ is unbounded, the classical Moment-SOS hierarchy typically does not converge.
There exist some works on solving polynomial optimization with unbounded sets.
Based on Karush-Kuhn-Tucker (KKT) conditions and Lagrange multipliers, Nie proposed tight Moment-SOS relaxations for solving \reff{1.1} \cite{nie2019tight}.
In \cite{jeyakumar2014polynomial}, the authors proposed Moment-SOS relaxations by adding sublevel set constraints. The resulting hierarchy of relaxations is also convergent under the Archimedeanness for the new constraints. Based on Putinar-Vasilescu's Positivstellensatz \cite{putinar1999positive,putinar1999solving}, Mai, Lasserre, and Magron \cite{nonneg} proposed a new hierarchy of Moment-SOS relaxations by adding a small perturbation to the objective, and convergence to a neighborhood of $f_{\min}$ was proved if the optimal value is achievable.
The complexity of this new hierarchy
was studied in \cite{pvrate}. Recently, a homogenized Moment-SOS hierarchy was proposed in \cite{huang2023homogenization} to solve polynomial optimization with unbounded sets employing homogenization techniques, and finite convergence was proved if standard optimality conditions hold at every global minimizer, including those at infinity. A theoretically interesting problem for polynomial optimization with unbounded sets is the case where the optimal value is not achievable. We refer to \cite{ha2009solving,huang2023homogenization, pham2023tangencies,schweighofer2006global,vui2008global} for related works.

A drawback of the Moment-SOS hierarchy is its limited scalability. This is because the size of involved matrices at the $k$th order relaxation is $\tbinom{n+k}{k}$ which increases rapidly as $n$, $k$ grow, and current semidefinite program (SDP) solvers based on interior-point methods can typically solve SDPs involving matrices of moderate size (say, $\leq 2,000$) in reasonable time on a standard laptop \cite{toh2018some}. An important way to improve the scalability is exploiting sparsity of inputting polynomials. There are two types of sparsity patterns in the literature to reduce the size of SDP relaxations: {\it correlative sparsity} and {\it term sparsity}. Correlative sparsity \cite{waki2006sums} considers the sparsity pattern of variables. The resulting sparse Moment-SOS hierarchy is obtained by building blocks of SDP matrices with respect to subsets of the input variables. Under the so-called running intersection property (RIP) and Archimedeanness, this sparse hierarchy was shown to have asymptotic convergence in \cite{grimm2007note,kojima2009note,lasserre2006convergent}. In contrast, term sparsity proposed by Wang et al. \cite{chordaltssos, tssos} considers the sparsity of monomials or terms. One can obtain a two-level block Moment-SOS hierarchy by using a two-step iterative procedure (a support extension operation followed by a block closure or chordal extension operation) to exploit term sparsity. For both types of sparsity, if the size of obtained SDP blocks is relatively small, then the resulting SDP relaxations are more tractable and computational costs can be significantly reduced. They have been successfully applied to solve optimal power flow problems \cite{josz2018lasserre,acopf}, round-off error bound analysis \cite{magron2018interval}, noncommutative polynomial optimization \cite{klep2021sparse,nctssos}, neural network verification \cite{newton2023sparse}, dynamical systems analysis \cite{sparsedynsys}, etc.
 
However, the above sparse Moment-SOS hierarchies may not converge when $K$ is unbounded as in the dense case. For the unbounded case, Mai, Lasserre, and Magron \cite{mai2023sparse} have recently provided a sparse version of Putinar–Vasilescu’s Positivstellensatz. To be more specific, it was proved that if the problem \reff{1.1} admits a correlative sparsity pattern $(\bx(1)\dots,\bx(p))$ satisfying the RIP and $f\geq 0$ on $K$, then for every $\epsilon>0$,
there exist sums of squares $\sigma_{0,\ell}$, $\sigma_{j,\ell}$, $j \in J_{\ell}$ of suitable degrees in variables $\bx(\ell)$, $\ell=1, \ldots, p$ such that
$$
f+\varepsilon \sum_{\ell=1}^p \bigg(1+\sum\limits_{x_i\in \bx(\ell)} x_i^2\bigg)^{d}=\sum_{\ell=1}^p \frac{\sigma_{0, \ell}+\sum_{j \in J_{\ell}} \sigma_{j, \ell} g_j}{\Theta_{\ell}^k},
$$ 
where $d \geq 1+\lfloor\operatorname{deg}(f) / 2\rfloor$ and $\Theta_{\ell}^k$, $\ell=1, \ldots, p$ are typically high-degree denominators (see Section \ref{pocs} for related notations and concepts).
Based on this, a sparse Moment-SOS hierarchy with perturbations is proposed to solve sparse polynomial optimization with unbounded sets. However, due to the occurrence of high-degree denominators, it is limited to solving problems with up to $10$ variables. The computational benefit of this sparse hierarchy is hence rather limited, and it is essentially a theoretical result as stated in \cite{mai2023sparse}.

\subsection*{Contributions} 
This paper studies sparse polynomial optimization with unbounded sets using homogenization techniques. Our new contributions are as follows.

\bit

\item [I.] When the problem \reff{1.1} admits correlative sparsity, we propose a sparse homogenized reformulation for \reff{1.1} while preserving the correlative sparsity pattern of the original problem. The sparse reformulation introduces two new types of variables. One is the homogenization variable, and the other consists of auxiliary variables associated to each variable clique. Then we apply the sparse Moment-SOS hierarchy to solve the new reformulation with a small perturbation. Under the RIP, we prove that the sequence of lower bounds produced by this hierarchy converges to a near neighborhood of $f_{\min}$.
\item [II.] To remove undesired perturbations, we also propose two alternative sparse homogenized reformulations of \reff{1.1} at the cost of possibly increasing the maximal clique size. We establish asymptotic convergence of the resulting sparse Moment-SOS hierarchies to $f_{\min}$. 
\item [III.] Based on the sparse homogenized reformulations, novel Positivstellens\"atze are provided for sparse positive polynomials on unbounded sets. 
\item [IV.] Diverse numerical experiments demonstrate that our approach performs much better than the usual sparse Moment-SOS hierarchy when solving sparse polynomial optimization on unbounded sets. In fact, with it we are able to handle such problems with up to thousands of variables!
\item [V.] To further illustrate its power, we apply our approach to trajectory optimization problems arising from the fields of robotics and control. It turns out that our approach can achieve global solutions for those problems with high accuracy.

\eit 

The rest of this paper is organized as follows.
Section~\ref{sc:pre} reviews some basics
about polynomial optimization.
Section~\ref{sc:shmg} introduces the sparse homogenized Moment-SOS hierarchy with perturbations and presents its asymptotic convergence result. Then Positivstellens\"atze with perturbations are provided.
In Section~\ref{sc:shmg2}, we introduce two alternative sparse homogenized Moment-SOS hierarchies without perturbations and prove their asymptotic convergence. Positivstellens\"atze without perturbations are provided.
Numerical experiments are presented in Section~\ref{num:ex}. Applications to trajectory optimization are provided in Section~\ref{tra:opt}. Section~\ref{sec:con} draws conclusions and make some discussions.

\section{Notations and preliminaries}\label{sc:pre}
\subsection*{Notation}
The symbol $\mathbb{N}$ (resp., $\mathbb{R}$) denotes the set of nonnegative integers (resp., real numbers). For $n\in\mathbb{N}$, let $[n]\coloneqq\{1,\dots, n\}$. Let $\bx\coloneqq (x_1,\dots,x_n)$ denote a tuple of variables and let $\bx^{2}\coloneqq(x_1^2,\dots,x_n^2)$. By slight abuse of notation, we also view $\bx$ as a set, i.e., $\bx=\{x_1,\dots,x_n\}$. For $\ba=(\alpha_1,\dots,\alpha_n)\in\N^n$, let
\[\bx^{\ba}\coloneqq x_1^{\alpha_1}\cdots x_n^{\alpha_n},~~\lvert\ba\rvert  \coloneqq \alpha_{1}+\cdots+\alpha_{n}.
\] 
For $k\in\N$, let $\mathbb{N}_{k}^{n}\coloneqq\left\{\ba \in \mathbb{N}^{n}\mid |\ba|\le k \right\}$. Denote by $[\bx]_k$  the vector of all monomials in $\bx$ with degrees $\leq k$, i.e.,
\[[\bx]_k\coloneqq [1,  x_1, x_2, \ldots, x_1^2, x_1x_2, \ldots,x_1^k, x_1^{k-1}x_2, \ldots, x_n^k ]^{\intercal}.\]
Let $\mathbb{R}[\bx] \coloneqq \mathbb{R}[x_1,\dots,x_n]$ be the ring of polynomials in $\bx$ with real coefficients, and $\mathbb{R}[\bx]_k\subseteq\re[\bx]$ is the subset of polynomials with degrees  $\leq k$. For a polynomial $p\in\re[\bx]$, denote by $\deg(p)$, $p^{(\infty)}$,  $\tilde{p}$ its total degree,  highest degree part and homogenization with respect to the homogenization variable $x_0$ (i.e., $\tilde{p}(\tilde{\bx})=x_0^{\deg(p)} p(\bx/x_0)$ with $\tilde{\bx} \coloneqq (x_0,x_1,\dots,x_n)$), respectively.
A homogeneous polynomial is said to be a \emph{form}.
A form $p$ is \emph{positive definite} if $p(\bx)>0$ for all nonzero $\bx \in \re^n$.
We write $A \succeq 0$ to indicate that a symmetric matrix $A$ is positive semidefinite. For a vector $\bv\in\R^n$, $\|\bv\|$ denotes the standard Euclidean norm. We write $\mathbf{0}$ (resp., $\mathbf{1}$) for the zero (resp., all-one) vector whose dimension is clear from the context. For $t\in\re$, $\lceil t\rceil$ denotes the smallest integer greater than or equal to $t$.

\subsection{Some basics for polynomial optimization}\label{ssc:pre:pop}
We review some basics in real algebraic geometry and polynomial optimization, referring to \cite{lasserre2015introduction,laurent2009sums,nie2023moment} for more details.

A subset $I\subseteq \mathbb{R}[\bx]$ is called an \emph{ideal} of $\re[\bx]$ if $I \cdot \mathbb{R}[\bx] \subseteq I$, $I+I \subseteq I$. For a polynomial tuple $h  \coloneqq (h_1,\dots, h_l)$, $\ideal{h}$ denotes the ideal generated by $h$, i.e.,
\begin{equation*}
\ideal{h} \coloneqq h_1 \cdot \mathbb{R}[\bx]+\cdots+h_l \cdot \mathbb{R}[\bx].
\end{equation*}
For $k\in\N$, the $k$th degree truncation of $\ideal{h}$ is
\begin{equation*}
\ideal{h}_{k} \coloneqq h_1 \cdot \mathbb{R}[\bx]_{k-\deg(h_1)}+\cdots+h_l \cdot \mathbb{R}[\bx]_{k-\deg(h_l)}.
\end{equation*}
Given a subset of variables $\bx^{\prime} \subseteq \bx$, if the polynomial tuple $h \in \mathbb{R}[\bx^{\prime}]^l$, 
we denote the ideal generated by $h$ in $ \mathbb{R}[\bx^{\prime}]$ by
\begin{equation*}
\ideal{h,\bx^{\prime} } \coloneqq h_1 \cdot \mathbb{R}[\bx^{\prime} ]+\cdots+
h_l \cdot \mathbb{R}[\bx^{\prime}].
\end{equation*}
Its $k$th degree truncation is defined as
\begin{equation*}
\ideal{h,\bx^{\prime} }_{k} \coloneqq h_1 \cdot \mathbb{R}[\bx^{\prime} ]_{k-\deg(h_1)}+\cdots+
h_l \cdot \mathbb{R}[\bx^{\prime}]_{k-\deg(h_l)}.
\end{equation*}

A polynomial $p\in\R[\bx]$ is said to be a \emph{sum of squares} (SOS) if $p=p_1^2+\dots+p_t^2$ for some $p_1,\dots,p_t \in \mathbb{R}[\bx]$. The set of all SOS polynomials in $\R[\bx]$ is denoted by $\Sigma[\bx]$. For $k\in\N$, let $\Sigma[\bx]_{k} \coloneqq\Sigma[\bx] \cap  \mathbb{R}[\bx]_{k}$.
For a polynomial tuple $g=(g_1,\dots,g_{m})$, the \emph{quadratic module} generated by $g$ is defined by
\be
\qmod{g} \coloneqq \Sigma[\bx]+ g_1 \cdot \Sigma[\bx]+\cdots+ g_{m} \cdot \Sigma[\bx].
\ee
For $k\in\N$, the $k$th degree truncation of $\qmod{g}$ is
\be
\qmod{g}_{k} \coloneqq \Sigma[\bx]_{k}+ g_1 \cdot \Sigma[\bx]_{k- \lceil \deg(g_1)/2 \rceil}+\cdots+ g_{m} \cdot \Sigma[\bx]_{k-\lceil \deg(g_{m}) /2\rceil}.
\ee
Similarly, if $g \in \mathbb{R}[\bx^{\prime}]^m$ for $\bx^{\prime} \subseteq \bx$, its quadratic module generated by $g$ in $\mathbb{R}[\bx^{\prime}]$ and $k$th degree truncation are denoted as
\be \nn
\qmod{g,\bx^{\prime}} \coloneqq \Sigma[\bx^{\prime}]+ g_1 \cdot \Sigma[\bx^{\prime}]+\cdots+ g_{m} \cdot \Sigma[\bx^{\prime}],
\ee
\[
\qmod{g,\bx^{\prime}}_{k} \coloneqq \Sigma[\bx^{\prime}]_{k}+ g_1 \cdot \Sigma[\bx^{\prime}]_{k-  \lceil \deg(g_1)/2 \rceil}+\cdots+ g_{m} \cdot \Sigma[\bx^{\prime}]_{k-\lceil \deg(g_{m}) /2\rceil}.
\]
The set $\ideal{h}+\qmod{g}$ is said to be \emph{Archimedean} if there exists $R>0$ such that $R-\|\bx\|^2\in\ideal{h}+\qmod{g}$. Clearly, if $p\in \ideal{h}+\qmod{g}$, then $p\ge0$ on the semialgebraic set $S\coloneqq\left\{\bx\in\mathbb{R}^{n}\mid h(\bx)=\mathbf{0}, g(\bx)\geq\mathbf{0}\right\}$ while the converse is not always true. However, if $p$ is positive on $S$ and $\ideal{h}+\qmod{g}$ is Archimedean, we have $p \in \ideal{h}+\qmod{g}$.
This conclusion is referred to as Putinar's Positivstellensatz \cite{putinar1993positive}.
	
For $k\in\N$, let $\re^{\N^n_{2k}}$ be the set of all real vectors that are indexed by $\N^n_{2k}$.
Given $\by\in\re^{\N^n_{2k}}$, define the following Riesz linear functional:
\be\label{<f,y>}
\langle p, \by\rangle\coloneqq\sum_{|\ba|\le 2k} p_{\ba} y_{\ba}, \quad\forall p=\sum_{|\ba|\le 2k} p_{\ba}\bx^{\ba}\in\re[\bx]_{2k}.
\ee
For a polynomial $p\in\re[\bx]$ and $\by\in\re^{\N^n_{2k+\deg(p)}}$, the $k$th \emph{localizing matrix} $M_k[p\by]$ associated with $p$ is the symmetric matrix indexed by $\N^n_{k}$ such that
\be \label{df:Lf[y]}
q^{\intercal} \Big( M_k[p\by] \Big) q =
\left\langle p (q^{\intercal}[\bx]_{k})^2, \by \right\rangle
\ee
for all $q \in \re^{\N^n_{k}}$.
In particular, if $p=1$, then $M_k[\by]$ is called the $k$th \emph{moment matrix}.
For $\bx^{\prime}\subseteq\bx$ and $p\in\mathbb{R}[\bx^{\prime}]$, let $M_k[p\by, \bx^{\prime}]$ be the localizing submatrix obtained by retaining only those rows and columns of $M_k[p\by]$ indexed by $\ba\in \mathbb{N}^n$ with $\bx^{\ba}\in\mathbb{R}[\bx^{\prime}]$.

\subsection{The homogenized Moment-SOS hierarchy}
When the feasible set $K$ is unbounded,
the standard Moment-SOS hierarchy 
typically fails to  have convergence. In this section, we present the homogenization approach introduced  \cite{huang2023homogenization} for solving polynomial optimization with unbounded sets.


Let $\tilde{\bx}=(x_0,x)\in \mR^{n+1}$. For the feasible set $K$ given in $\eqref{1.1}$, define the  homogenized set 
\be \label{sets:KKK}
\widetilde{K}\coloneqq\left\{\tilde{\bx}\in\mathbb{R}^{n+1}
\left|\baray{l}
\tilde{g}_{j}(\tilde{\bx}) \geq 0,\quad j\in [m],\\
x_0\geq 0,\quad\|\tilde{\bx}\|^2=1.
\earay \right. \right\}
\ee
The set $K$ is said to be \emph{closed at infinity} if
\[
\widetilde{K}=\cl{\widetilde{K}\cap \{\tilde{\bx}\in\mathbb{R}^{n+1}\mid x_0>0\}},
\]
where $\cl{\cdot}$ is the closure operator.
A basic property for closedness at infinity is that $f-\gamma \ge 0$ on $K$ if and only if
$\tilde{f}(\tilde{\bx}) - \gamma x_0^d \ge 0$ on $\widetilde{K}$ with $d\coloneqq\deg(f)$ \cite{huang2023homogenization}.
Therefore, when $K$ is closed at infinity,  \reff{1.1} is equivalent to the following homogenized optimization problem:
\begin{equation}\label{hpop}
\left\{\baray{lll}
\sup&\gamma\\
\st&\tilde{f}(\tilde{\bx}) - \gamma x_0^d \geq 0\text{ on }\widetilde{K}.
\earay \right.
\end{equation}
By applying the standard Moment-SOS relaxations to solve the homogenized reformulation \reff{hpop}, a homogenized Moment-SOS hierarchy was proposed in
\cite{huang2023gmp,huang2023homogenization} to solve $\eqref{1.1}$ with unbounded  sets. Asymptotic and finite convergences were proved under some generic assumptions.

\subsection{Polynomial optimization with correlative sparsity}\label{pocs}
The Moment-SOS hierarchy with correlative sparsity was first studied in \cite{waki2006sums}.
Suppose that the subsets of variables $\bx(1)$, $\dots$, $\bx(p) \subseteq \bx$ satisfy $\cup_{\ell=1}^p\bx(\ell)=\bx$. The POP \reff{1.1} is said to have a \emph{correlative sparsity pattern} (csp) $(\bx(1),\dots, \bx(p))$ if 
\bit
\item [(i)] The objective function $f\in\mathbb{R}[\bx]$ can be written as 
$$
f=\sum_{\ell=1}^p f_\ell \text{ with } f_\ell \in \mathbb{R}[\bx(\ell)] \text{ for } \ell\in[p];
$$
\item [(ii)] There exists a partition $\{J_1,\dots, J_p\}$ 
of $[m]$ such that for every $\ell\in [p]$ and every $j \in J_\ell$, we have $g_j \in \mathbb{R}[\bx(\ell)]$.
\eit

Let
$$d_{\min}\coloneqq\max\{\lceil\deg(f)/2\rceil,\lceil\deg(g_{1})/2\rceil,\ldots,\lceil\deg(g_{m})/2\rceil\}.$$
Given $k\ge d_{\min}$, the $k$th order SOS relaxation with correlative sparsity for \reff{1.1} is 
\begin{equation} \label{spasos}  
\left\{ \begin{array}{lll}
\sup &   \gamma \\
\st &  f - \gamma \in \qmod{(g_j)_{j\in J_1},\bx(1)}_{2k}+\cdots+\qmod{(g_j)_{j\in J_p},\bx(p)}_{2k}.
\end{array} \right.
\end{equation}
The dual of \reff{spasos} is the $k$th order moment relaxation
\be \label{spamom}
\left\{\baray{ll}
\inf  &\langle f,\by \rangle  \\
\st  &y_{\mathbf{0}}=1, \quad M_k\left[\by, \bx(\ell)\right] \succeq 0, \quad\ell \in [p],\\
& M_{k-\lceil\deg(g_{j})/2\rceil}\left[g_j\by, \bx(\ell)\right]  \succeq 0,\quad j \in J_\ell,  \ell \in [p]. \\
\earay\right.
\ee

The csp $(\bx(1),\dots,\bx(p))$ is said to satisfy \emph{running intersection property} (RIP) if for every $\ell\in [p-1]$, there exists some $s\in [\ell]$ such that 
\begin{equation} \label{rip}
\bx(\ell+1) \cap \bigcup_{j=1}^\ell \bx(j) \subseteq \bx(s).
\end{equation}
Under the RIP, it was shown in 
\cite{grimm2007note,kojima2009note,lasserre2006convergent} that the sparse Moment-SOS hierarchy \reff{spasos}--\reff{spamom} has asymptotic convergence.
\begin{thm}[\cite{grimm2007note,kojima2009note,lasserre2006convergent}]
\label{spaput}
Suppose that  \reff{1.1} has the csp $(\bx(1),\dots,\bx(p))$ satisfying the RIP, and the quadratic module $\qmod{(g_j)_{j\in J_{\ell}},\bx(\ell)}$ is Archimedean for each $\ell \in [p]$. 
If $f>0$ on $K$, then
\[f \in \qmod{(g_j)_{j\in J_1},\bx(1)}+\cdots+\qmod{(g_j)_{j\in J_p},\bx(p)}.\]
\end{thm}

\begin{remark}
For a given POP, a csp $(\bx(1),\dots,\bx(p))$ satisfying the RIP can be obtained by: (1) building the csp graph, (2) generating a chordal extension,  (3) taking the list of maximal cliques; see \cite{magron2023sparse,waki2006sums}. Note that chordal extensions of a graph are typically not unique and so are correlative sparsity patterns.
\end{remark}

\section{The sparse homogenized Moment-SOS hierarchy with perturbations}\label{sc:shmg}
In this section, we give a hierarchy of sparse homogenized  Moment-SOS relaxations to solve polynomial optimization with unbounded sets.
Suppose the POP \reff{1.1} admits a correlative sparse pattern and  $K$ is unbounded. 
Note that we can not directly apply the sparse relaxations \reff{spasos}--\reff{spamom} to solve the homogenized reformation \eqref{hpop} since the spherical constraint $\|\tilde{\bx}\|^2=1$ destroys the csp of \reff{1.1}.
To overcome this difficulty, we give a sparse homogenized reformulation for \reff{1.1} by introducing a tuple of auxiliary variables.

Let $(\bx(1),\dots,\bx(p))$ be a csp of \reff{1.1} and $\{J_1,\dots,J_p\}$ a partition of $[m]$ such that for every $\ell\in [p]$ and every $j \in J_\ell$, $g_j \in \R[\bx(\ell)]$.
Define the sparse set
\be \label{spK}
\widetilde{K}_s^{1}\coloneqq\left\{(\tilde{\bx},\w)\in\mathbb{R}^{n+1+p}
\left| \baray{l}
x_{0} \geq 0,\quad\tilde{g}_{j}(\tilde{\bx}) \geq 0,\quad j\in [m], \\
\|\tilde{\bx}(\ell)\|^2+w_{\ell}^2=1,\quad\ell \in [p],\\
\earay \right. \right\}
\ee
where $\tilde{\bx}(\ell)\coloneqq(x_{0},\bx(\ell))$ and $\w\coloneqq(w_1,\ldots,w_p)$ is a tuple of auxiliary variables.
The difference between $\widetilde{K}_s^{1}$ and $\widetilde{K}$ is that we replace the single non-sparse spherical constraint $\|\tilde{\bx}\|^2=1$ by multiple spherical constraints $\|\tilde{\bx}(\ell)\|^2+w_{\ell}^2=1$, $\ell\in [p]$. Interestingly, $\widetilde{K}_s^{1}$ retains the correlative sparse pattern of the feasible set K.

Let $d\coloneqq\deg(f)$, $d_0\coloneqq2\lceil \frac{d}{2} \rceil$. Consider the sparse homogenized reformulation for \reff{1.1} with  perturbations:
\begin{equation}\label{hommax} 
\left\{\begin{array}{rl}
\sup&\gamma\\
\st&\tilde{f}(\tilde{\bx})+\epsilon 
\cdot \left(\sum\limits_{i=0}^n x_{i}^{d_0}+\sum\limits_{\ell=1}^p w_{\ell}^{d_0}\right)-\gamma x_0^d\geq0, ~\forall~(\tilde{\bx},\w)\in \widetilde{K}_s^{1},
\end{array}\right.
\end{equation}
where $\epsilon\geq 0$ is a tunable parameter. Let $f^{(\epsilon)}$ be the optimal value of \reff{hommax} and
\[ h_\ell\coloneqq\|\tilde{\bx}(\ell)\|^2+w_{\ell}^2-1,\quad\ell\in [p].
\]

For a relaxation order $k\ge d_{\min}$, the $k$th sparse homogenized SOS relaxation for \reff{hommax} is
\begin{equation}\label{spsos}
\left\{ \begin{array}{lll}
\sup&\gamma\\
\st&\tilde{f}(\tilde{\bx}) +\epsilon \cdot 
\left(\sum\limits_{i=0}^n x_{i}^{d_0}+\sum\limits_{\ell=1}^p w_{\ell}^{d_0}\right) - \gamma x_0^d \\
&\quad\quad\in\sum\limits_{\ell=1}^p \big(\ideal{h_\ell,\tilde{\bx}(\ell)\cup\{w_{\ell}\}}_{2k}+\qmod{\{x_0\}\cup\{\tilde{g}_j\}_{j\in J_\ell},\tilde{\bx}(\ell)\cup\{w_{\ell}\}}_{2k}\big).
\end{array}\right.
\end{equation}
The dual of \reff{spsos} is the $k$th sparse homogenized moment relaxation:
\be\label{spmom}
\left\{\begin{array}{lll}
\sup&\langle \tilde{f}(\tilde{\bx}) +\epsilon \cdot \left(\sum\limits_{i=0}^n x_{i}^{d_0}+\sum\limits_{\ell=1}^p w_{\ell}^{d_0}
\right), \by\rangle & \\
\st &\langle x_0^d,\by\rangle=1, \quad M_{k-1}[h_\ell \by, \tilde{\bx}(\ell)\cup\{w_{\ell}\}] =0, \quad\ell \in [p], \\
& M_{k-1}[x_0, \tilde{\bx}(\ell)\cup\{w_{\ell}\}] \succeq 0,\quad M_k[\by, \tilde{\bx}(\ell)\cup\{w_{\ell}\}] \succeq 0, \quad \ell \in [p],\\
& M_{k-\lceil\deg(g_{j})/2\rceil}[\tilde{g}_j\by, \tilde{\bx}(\ell)\cup\{w_{\ell}\}] \succeq 0, \quad j\in J_\ell, \ell\in [p].
\end{array}\right.
\ee
The hierarchy of relaxations \reff{spsos}--\reff{spmom} is called the sparse homogenized Moment-SOS hierarchy for solving \reff{1.1}.
Let $f_k$, $f_k^{\prime}$ denote the optima of \reff{spsos} and \reff{spmom}, respectively.

\subsection{Convergence analysis}\label{subse:conana}
When $K$ is closed at infinity, we establish the relationship between the optimal values of \reff{1.1} and \reff{hommax}. 
\begin{theorem}\label{optval:err}
Suppose that $K$ is closed at infinity and $\bx^*$ is a minimizer of \reff{1.1}. Let $f^{(\epsilon)}$ be the optimal value of \reff{hommax}.
For every $\epsilon>0$, the following holds
\[f_{\min}< f^{(\epsilon)}\leq f_{\min}+ \epsilon \cdot p \cdot \left(1+\|\bx^*\|^2\right)^{\frac{d}{2}}.
\]
Moreover, one has $f^{(\epsilon)}=f_{\min}$ when $\epsilon =0$.
\end{theorem}
\begin{proof}
Since $K$ is closed at infinity and $f-f_{\min}\geq0 $ on $K$, we have $\tilde{f}-f_{\min}x_0^d\geq 0$ on $\widetilde{K}$. Note that $\mathbf{0}\notin \widetilde{K}_s^{1}$. Hence, for every $(\tilde{\bx},\w)\in \widetilde{K}_s^{1}$, we have
\[
\tilde{f}(\tilde{\bx})- f_{\min} x_0^d +\epsilon \cdot \left(\sum\limits_{i=0}^n x_{i}^{d_0}+\sum\limits_{\ell=1}^p w_{\ell}^{d_0}\right)>0,
\] 
which implies $f_{\min}< f^{(\epsilon)}$. On the other hand, let
\[\tilde{\bx}^*\coloneqq\frac{(1,\bx^*)}{\sqrt{1+\|\bx^*\|^2}},~ w_{\ell}^*\coloneqq\sqrt{1-\|\tilde{\bx}^*(\ell)\|^2}~(\ell \in [p]).
\]
Then it holds
\[\tilde{g}_j(\tilde{\bx}^*)=g_j(\bx^*)/(\sqrt{1+\|\bx^*\|^2})^{\deg(g_j)}\geq  0,\quad j\in [m],\]	
and so $(\tilde{\bx}^*,\w^*)\in \widetilde{K}_s^{1}$. If $\gamma$ is a feasible point of \eqref{hommax}, we have (noting $d_0\geq 2$)
\begin{equation*}
\gamma\cdot (\tilde{x}^*_0)^d\le\tilde{f}(\tilde{\bx}^*)+\epsilon \cdot\left(\sum\limits_{i=0}^n (\tilde{x}^*_{i})^{d_0}+\sum\limits_{\ell=1}^p (w^*_{\ell})^{d_0} \right)\le (\tilde{x}^*_0)^df_{\min}+\epsilon \cdot p.
\end{equation*}
It follows $f^{(\epsilon)}\le f_{\min}+\epsilon  \cdot p/(\tilde{x}^*_0)^d=f_{\min}+ \epsilon \cdot p \cdot\left(1+\|\bx^*\|^2\right)^{\frac{d}{2}}$.
When $\epsilon=0$, the above implies that $\gamma \leq f_{\min}$ for every feasible point $\gamma$ of \eqref{hommax}. Thus, we know that $f^{(0)}=f_{\min}$.
\end{proof}

\begin{remark}
Auxiliary variables $\w$ are necessary for Theorem \ref{optval:err} to be true. For instance, let $\bx(1)=\{x_1,x_2\}$, $\bx(2)=\{x_2,x_3\}$. Consider the unconstrained optimization problem with $f=x_1^2x_2^4-x_2^4x_3^2$. Clearly, we have $f_{\min}=-\infty$. In this case, the sparse homogenized reformulation \reff{hommax} without auxiliary variables $\w$ reads as 
\begin{equation} \label{locl:per}
\left\{\begin{array}{rl}
\sup& \gamma \\
\st &  x_1^2x_2^4-x_2^4x_3^2 +\epsilon \cdot(x_0^4+x_1^4+x_2^4) - \gamma x_0^4 \geq 0, ~\forall\bx\in \widetilde{K}_s^{1},
\end{array} \right.
\end{equation}
where 
\[
\widetilde{K}_s^{1}=\{(x_0,x_1,x_2,x_3)\in \mR^4:x_0^2+x_1^2+x_2^2=1,x_0^2+x_2^2+x_3^2=1\}.
\]
For arbitrary $\epsilon \geq 0$, the optimal value of \reff{locl:per} is nonnegative. So Theorem \ref{optval:err} fails.
\end{remark}

The next lemma shows that \reff{hommax} indeed inherits the csp of the original problem \reff{1.1}.
\begin{lemma}\label{keepcsp}
Suppose that \reff{1.1} admits a csp $(\bx(1),\dots,\bx(p))$. Then 
\reff{hommax} admits the csp $(\tilde{\bx}(1)\cup\{w_1\},\dots,\tilde{\bx}(p)\cup\{w_p\})$. Furthermore, if $(\bx(1),\dots,\bx(p))$ satisfies the RIP, so does $(\tilde{\bx}(1)\cup\{w_1\},\dots,\tilde{\bx}(p)\cup\{w_p\})$.
\end{lemma}
\begin{proof}
 By the assumption, we know that $f=f_1+\cdots+f_p\in\mR[\bx]$ with $f_{\ell}\in\mR[\bx(\ell)] $ $(\ell\in [p])$. For $\ell=1,\dots,p$, we have that 
 \[
 \tilde{f}_{\ell}\in \mathbb{R}[\tilde{\bx}(\ell)\cup\{w_{\ell}\}], ~h_{\ell}\in \mathbb{R}[\tilde{\bx}(\ell)\cup\{w_{\ell}\}], ~\tilde{g}_j \in \mathbb{R}[\tilde{\bx}(\ell)\cup\{w_{\ell}\}]~(j \in J_\ell).
 \]
 Thus,
\reff{hommax} admits the csp $(\tilde{\bx}(1)\cup\{w_1\},\dots,\tilde{\bx}(p)\cup\{w_p\})$. 
 Then the conclusion follows.
 For every $\ell\in [p-1]$, it holds that
 \begin{equation} 
 \baray{l}
(\tilde{\bx}(\ell+1)\cup\{w_{\ell+1}\})\cap (\bigcup_{j=1}^\ell	\tilde{\bx}(j)\cup\{w_{j}\})\\
\quad \quad \quad \quad \quad \quad \quad=(x(\ell+1)\cup \{x_0,w_{\ell+1} \})\cap \bigcup_{j=1}^\ell(x(j)\cup \{x_0,w_{j} \}),\\
 \quad \quad \quad \quad \quad \quad \quad=\{x_0\}\cup (I_{\ell+1}\cap \bigcup_{j=1}^\ell I_{j}). \\
 \earay
 \end{equation}	\nonumber
 Hence, the conclusion follows.
\end{proof}

In the following, we prove that  the sparse homogenized Moment-SOS hierarchy \reff{spsos}-\reff{spmom} has asymptotic convergence to a  neighbourhood of $f_{\min}$.
\begin{thm}
Suppose that $K$ is closed at infinity and POP \reff{1.1} admits a csp $(\bx(1),\dots,\bx(p))$ satisfying the RIP. Then for every $\epsilon>0$, we have $f_k\rightarrow f^{(\epsilon)}$ as $k\rightarrow \infty$.
\end{thm}

\begin{proof}
For $\gamma<f^{(\epsilon)}$, we show that for each $(\tilde{\bx},\w)\in \widetilde{K}_s^{1}$, 
\be \label{eq:pos}
\tilde{f}(\tilde{\bx}) +\epsilon \cdot\left(\sum\limits_{i=0}^n x_{i}^{d_0}+\sum\limits_{\ell=1}^p w_{\ell}^{d_0}\right)-\gamma x_0^d>0.
\ee
If $x_0=0$, then 
$$\tilde{f}(\tilde{\bx})+\epsilon\cdot \left(\sum\limits_{i=0}^n x_{i}^{d_0}+\sum\limits_{\ell=1}^p w_{\ell}^{d_0}\right)-\gamma x_0^d\geq \epsilon \cdot\left(\sum\limits_{i=1}^n x_{i}^{d_0}+\sum\limits_{\ell=1}^p w_{\ell}^{d_0}\right)>0.$$
If $x_0\neq0$, we have
\begin{align*}
&\tilde{f}(\tilde{\bx})+\epsilon \cdot 
\left(\sum\limits_{i=0}^n x_{i}^{d_0}+\sum\limits_{\ell=1}^p w_{\ell}^{d_0}\right)-\gamma x_0^d\\
=\,&\tilde{f}(\tilde{\bx})+\epsilon \cdot\left(\sum\limits_{i=0}^n x_{i}^{d_0}+\sum\limits_{\ell=1}^p w_{\ell}^{d_0}\right)-f^{(\epsilon)} x_0^d+(f^{(\epsilon)}-\gamma)x_0^d>0.
\end{align*}
Thus, we have
$\tilde{f}(\tilde{\bx})+\epsilon \cdot\left(\sum\limits_{i=0}^n x_{i}^{d_0}+\sum\limits_{\ell=1}^p w_{\ell}^{d_0}\right)-\gamma x_0^d>0$ on $\widetilde{K}_s^{1}$ for any $\gamma<f^{(\epsilon)}$. The spherical constraint $h_\ell=0$ implies that $\ideal{h_\ell,\tilde{\bx}(\ell)\cup\{w_{\ell}\}}+ \qmod{(\tilde{g}_j)_{j\in J_\ell},\tilde{\bx}(\ell)\cup\{w_{\ell}\}}$ is Archimedean in $\mathbb{R}[\tilde{\bx}(\ell)\cup\{w_{\ell}\}]$ for each $\ell\in[p]$. By Lemma \ref{keepcsp}, POP \reff{hommax} admits the csp $(\tilde{\bx}(1)\cup\{w_1\},\dots,\tilde{\bx}(p)\cup\{w_p\})$ satisfying the RIP. It  follows from Theorem \ref{spaput} that
\begin{align*}
\tilde{f}(\tilde{\bx}) +\epsilon\cdot\left(\sum\limits_{i=0}^n x_{i}^{d_0}+\sum\limits_{\ell=1}^p w_{\ell}^{d_0}\right)-\gamma x_0^d \in &\sum\limits_{\ell=1}^p (\ideal{h_\ell,\tilde{\bx}(\ell)\cup\{w_{\ell}\}}\\
&+ \qmod{\{x_0\} \cup \{\tilde{g}_j\}_{j\in J_\ell},\tilde{\bx}(\ell)\cup\{w_{\ell}\}}).
\end{align*}
Thus we obtain $f_k\rightarrow f^{(\epsilon)}$ as $k\rightarrow \infty$.
\end{proof}

\begin{remark}\label{rema3.5}
If there is no perturbation, i.e., $\epsilon=0$, the hierarchy \reff{spsos}--\reff{spmom} still provides valid lower bounds to $f_{\min}$. However, the convergence to $f_{\min}$ may not happen.
For instance, let
\[
\bx(1)=\{x_1,x_2,x_3,x_4,x_5\},\quad\bx(2)=\{x_5,x_6,x_7\},
\]
and consider the unconstrained optimization problem $\inf_{\bx\in \mR^7} f(\bx)$ with  $f=f_1+f_2$, where 
\begin{align*}
f_1&=(x_4^2+x_5^2+1)\left(x_1^4 x_2^2+x_2^4 x_3^2+x_1^2 x_3^4-3 x_1^2 x_2^2 x_3^2\right)+x_3^8,\\
f_2&=x_5^2x_6^2x_7^2.
\end{align*}
We show that for arbitrary $\gamma <f_{\min}=0$, it holds that
\[
\tilde{f}-\gamma x_0^8 \notin \sum\limits_{\ell=1}^2 (\ideal{\|\tilde{\bx}(\ell)\|^2+w_{\ell}^2-1, \tilde{\bx}(\ell)\cup\{w_{\ell}\}}+\qmod{x_0,\tilde{\bx}(\ell)\cup\{w_{\ell}\}}).
\]
Suppose otherwise that there were $h_{\ell}\in \mR[\tilde{\bx}(\ell)\cup\{w_{\ell}\}]$, $\sigma_{\ell}\in \qmod{x_0,\tilde{\bx}(\ell)\cup\{w_{\ell}\}})$ such that 
\begin{equation}\label{ex:eq1}
\tilde{f}-\gamma x_0^8 = \sum\limits_{\ell=1}^2 (h_{\ell}\cdot (\|\tilde{\bx}(\ell)\|^2+w_{\ell}^2-1)+\sigma_{\ell}).    
\end{equation}
Substituting $(0,0,1,0,0,0)$ for $(x_0,x_5,x_6,x_7,w_1,w_2)$ in \eqref{ex:eq1},  we obtain
\begin{equation}\label{ex:eq2}
x_4^2(x_1^4 x_2^2+x_2^4 x_3^2+x_1^2 x_3^4-3 x_1^2 x_2^2 x_3^2)+x_3^8 =\sigma_0+h_0\cdot(x_1^2+\cdots+x_4^2-1)
\end{equation}
for some $\sigma_0\in \Sigma[x_1,x_2,x_3,x_4]$, $h_0\in \mathbb{R}[x_1,x_2,x_3,x_4]$.
However, the left-hand side of \eqref{ex:eq2} is the dehomogenized Delzell's polynomial and the above representation does not exist as shown in \cite{reznick2000some}.
\end{remark}

\subsection{Sparse Positivstellens\"atze with perturbations}\label{spapos}
In this subsection, we provide new sparse Positivstellens\"atze for nonnegative polynomials on unbounded semialgebraic sets, based on sparse homogenized reformulations. Our Positivstellens\"atze do not need any denominator, and thus are quite different from the sparse versions of Reznick’s Positivstellensatz and Putinar–Vasilescu’s Positivstellensatz given in \cite{mai2023sparse}.


First, we consider the homogeneous case with $K=\mR^n$.
\begin{theorem}\label{sparserez}
Let $f\in\mR[\bx]$ be a form of degree $d$. Suppose $f$ admits a csp $(\bx(1),\dots,\bx(p))$ satisfying the RIP.
\bit 
\item[(1)] If $f\geq 0$ on $\mR^n$, then for any $\epsilon>0$, there exist  $\sigma_{\ell}\in\Sigma[\bx(\ell)\cup\{w_{\ell}\}]$, $\tau_{\ell}\in\mR[\bx(\ell)\cup\{w_{\ell}\}]$, $\ell \in [p]$ such that
\begin{equation}
f+\epsilon  \cdot \left(\sum\limits_{i=1}^n x_{i}^{d}+\sum\limits_{\ell=1}^p w_{\ell}^{d}\right)=\sum_{\ell=1}^p\left(\sigma_{\ell}+\tau_{\ell}\left(\|\bx(\ell)\|^2+w_{\ell}^2-1\right)\right),
\end{equation}
	
\item[(2)] If $f$ is positive definite, then for any $\epsilon>0$, there exist  $\sigma_{\ell}\in\Sigma[\bx(\ell)\cup\{w_{\ell}\}]$, $\tau_{\ell}\in\mR[\bx(\ell)\cup\{w_{\ell}\}]$, $\ell\in[p]$ such that
\begin{equation}
f+\epsilon \cdot \sum\limits_{\ell=1}^p w_{\ell}^{d}=\sum_{\ell=1}^p\left(\sigma_{\ell}+\tau_{\ell}\cdot \left(\|\bx(\ell)\|^2+w_\ell^2-1\right)\right),
\end{equation}
\eit
\end{theorem}

\begin{proof}
Note that $f\geq 0$ on $\mR^n$ is equivalent to  $f\geq 0$ on $S\coloneqq\{(\bx,\w)\in\mR^{n+p}\mid\|\bx(\ell)\|^2+w_{\ell}^2=1,\ell\in[p]\}$. Sicen $0\notin S$, we have $f+\epsilon (\sum\limits_{i=1}^n x_{i}^{d}+\sum\limits_{\ell=1}^p w_{\ell}^{d})>0$ on $S$. If $f$ is positive definite and there exists $(\bx,\w)\in \mR^{n+p}$ such that $f(\bx)+\epsilon\sum\limits_{\ell=1}^p w_{\ell}^{d}=0$, we must have $(\bx,\w)=\mathbf{0}$. Hence, we  have $ f+\epsilon\sum\limits_{\ell=1}^p w_{\ell}^{d}>0$ on $S$. Under given assumptions, items (i), (ii)  follow  from Theorem \ref{spaput}.
\end{proof}

Note that a polynomial $f\geq 0$ on $\mR^n$ is equivalent to $\tilde{f}\geq 0$ on $\mR^{n+1}$.
Theorem \ref{sparserez} can be  generalized to non-homogeneous polynomials directly. We omit the proof for cleanness.
\begin{theorem}
Let $f\in\mR[\bx]$ with $\deg(f)=d$. Suppose  that $f$ admits a csp $(\bx(1),\dots,\bx(p))$ satisfying the RIP.
\bit 
\item[(1)] If $f\geq 0$ on $\mathbb{R}^n$, then for any $\epsilon>0$, there exist $\sigma_{\ell}\in\Sigma[\tilde{\bx}(\ell)\cup\{w_\ell\}],\tau_{\ell}\in\mR[\tilde{\bx}(\ell)\cup\{w_\ell\}]$, $\ell\in [p]$ such that
\begin{equation}
\tilde{f}+\epsilon \left(\sum\limits_{i=0}^n x_{i}^{d}+\sum\limits_{\ell=1}^p w_{\ell}^{d}\right)=\sum_{\ell=1}^p\left(\sigma_{\ell}+\tau_{\ell}\left(\|\tilde{\bx}(\ell)\|^2+w_\ell^2-1\right)\right),
\end{equation}
	
\item[(2)] If $f>0$ on $\mathbb{R}^n$ and the form $f^{(\infty)}$ is positive definite, then for any $\epsilon>0$, there exist $\sigma_{\ell}\in\Sigma[\tilde{\bx}(\ell)\cup\{w_\ell\}]$, $\tau_{\ell}\in\mR[\tilde{\bx}(\ell)\cup\{w_\ell\}]$, $\ell \in [p]$ such that
\begin{equation}
\tilde{f}+\epsilon\sum\limits_{\ell=1}^p w_{\ell}^{d}=\sum_{\ell=1}^p\left(\sigma_{i}+\tau_{\ell}\left(\|\tilde{\bx}(\ell)\|^2+w_\ell^2-1\right)\right),
\end{equation}
\eit
\end{theorem}

Let $K$ be defined as in \reff{1.1}. Define 
\[K^{(\infty)}  \coloneqq  \left\{  \bx
\in \mathbb{R}^{n}
\left| \baray{l}
g^{(\infty)}_{j}(\bx) \geq 0,\quad j\in [m], \\
\|\bx\|^2-1=0.
\earay \right. \right\}\]
If $K$ is closed at infinity and $f$ is bounded from below on $K$, then $f^{(\infty)}\geq 0$ on $K^{(\infty)}$ \cite[Theorem 3.6.]{huang2023homogenization}. The polynomial $f$ is said to be \emph{positive at infinity} on $K$ if $f^{(\infty)}> 0$ on $K^{(\infty)}$. When $K$ is closed at infinity, $f\geq 0$ on $K$ if and only if $\tilde{f}\geq 0$ on $\widetilde{K}$. Thus we can derive the following sparse homogenized version of Putinar-Vasilescu's Positivstellensatz from Theorem \ref{spaput}.
\begin{theorem} \label{spa:PV}
Suppose that $K$ is closed at infinity, and  POP \reff{1.1} admits a csp $(\bx(1),\dots,\bx(p))$ satisfying the RIP.    
\bit
\item[(1)] If $f\geq 0$ on $K$, then for any $\epsilon>0$,  there exist $\sigma_{\ell}\in\qmod{(\{x_0\}\cup \{\tilde{g}_j\}_{j\in J_\ell},\tilde{\bx}(\ell)\cup\{w_{\ell}\}},\tau_{\ell}\in\mR[\tilde{\bx}(\ell)\cup\{w_{\ell}\}]$, $\ell\in [p]$ such that
\begin{equation}
\tilde{f}+\epsilon \left (\sum\limits_{i=0}^n x_{i}^{d_0}+\sum\limits_{\ell=1}^p w_{\ell}^{d_0}\right)=\sum_{\ell=1}^p\left(\sigma_{\ell}+\tau_{\ell}\left(\|\tilde{\bx}(\ell)\|^2+w_\ell^2-1\right)\right).
\end{equation}
\item[(2)] If $f>0$ on $K$ and $f$ is positive definite at infinity on $K$, then for any $\epsilon>0$, there exist $\sigma_{\ell}\in\qmod{(\{x_0\}\cup \{\tilde{g}_j\}_{j\in J_\ell},\tilde{\bx}(\ell)\cup\{w_{\ell}\}},\tau_{\ell}\in\mR[\tilde{\bx}(\ell)\cup\{w_{\ell}\}]$, $\ell\in [p]$ such that
\begin{equation}
\tilde{f}+\epsilon\sum\limits_{\ell=1}^p w_{\ell}^{d_0}=\sum_{\ell=1}^l\left(\sigma_{\ell}+\tau_{\ell}\left(\|\tilde{\bx}(\ell)\|^2+w_\ell^2-1\right)\right).
\end{equation}
\eit 
\end{theorem}

\subsection{Extraction of minimizers}\label{extr:mini}

In the case of the dense Moment-SOS hierarchy, a convenient criterion for detecting global optimality is flat extension/truncation (see \cite{curto2005truncated,laurent2009sums,nie2023moment}). A procedure for extracting minimizers is given in \cite{henrion2005detecting}. This procedure was generalized to polynomial optimization with correlative sparsity in \cite{lasserre2006convergent}. We  adapt it  to extract minimizers from the sparse homogenized moment relaxtions \reff{spmom}.

Let
\begin{equation}
d_{K}\coloneqq\max\{\lceil\deg(g_j)/2\rceil,j\in [m]\} \text{ and } d_{\ell}\coloneqq\max\{\lceil\deg(g_j)/2 \rceil,j\in J_{\ell}\}, \,\ell\in[p].
\end{equation}
Suppose that $\by^{*}$ is an optimal solution of $\eqref{spmom}$ at the $k$th order relaxation.
If there exists an integer $t\in[d_K, k]$ such that
\begin{equation}\label{FLT}
\begin{split}
\operatorname{rank} M_t\left(\by^{*}, \tilde{\bx}(\ell)\cup\{w_{\ell}\}\right)=\operatorname{rank} M_{t-d_{\ell}}\left(\by^{*}, \tilde{\bx}(\ell)\cup\{w_{\ell}\}\right), \text { for all } \ell \in[p],\\
\operatorname{rank} M_t\left(\by^{*}, \tilde{\bx}(i) \cap \tilde{\bx}(j)\right)=1, \text { for all }  i\ne j\in[p]  \text { with }  \tilde{\bx}(i) \cap \tilde{\bx}(j)\neq \emptyset,
\end{split}
\end{equation}
then the moment relaxation \reff{spmom} is exact, i.e., $f_k^{\prime}=f^{(\epsilon)}$. Furthermore, by applying the extraction procedure in \cite{henrion2005detecting} to each moment matrix $M_{t}(\by^*, \tilde{\bx}(\ell)\cup\{w_{\ell}\})$, we obtain a set of points
\[
\Delta_{\ell}\coloneqq\left\{(\tilde{\bx}^*(\ell),w^*_{\ell})\right\}\subseteq\mathbb{R}^{|\bx(\ell)|+2},\quad \ell\in [p],
\]
where $|\bx(\ell)|$ stands for the dimension of $\bx(\ell)$.
Next we show that approximate minimizers (exact minimizers if $\epsilon=0$) of \reff{1.1} can be extracted from the sets $\Delta_{\ell}$ $(\ell \in[p])$. 

\begin{thm}
Suppose that $K$ is closed at infinity  and POP \reff{1.1} admits a csp $(\bx(1),\dots,\bx(p))$, and $\by^{*}$ is an optimal solution of $\eqref{spmom}$ satisfying \reff{FLT}. Let
\begin{equation*}
\Omega\coloneqq\left\{(x_0^*,\bx^*)\in \mR^{n+1}\mid \text{there exists~} w^*_{\ell}\in \mR \text{~such that~} (\tilde{\bx}^*(\ell),w^*_{\ell}) \in \Delta_{\ell} ~\text{for~} \ell \in [p]\right\}.
\end{equation*}
If $\tilde{\bx}^*\in \Omega$ with $x_0^*>0$, then $\bx^*/x_0^*\in K$ and
\begin{equation}\label{sec4:eq1}
  f(\bx^*/x_0^*) +\epsilon \left(\sum\limits_{i=0}^n (x^*_{i})^{d_0}+\sum\limits_{\ell=1}^p (w^*_{\ell})^{d_0}\right)/(x_0^*)^d=f^{(\epsilon)},  
\end{equation}
where $w^*_{\ell}$ is from $\Delta_{\ell}$. In particular, when $\epsilon=0$, we have $f(\bx^*/x_0^*)=f_{\min}$ and $\bx^*/x_0^*$ is a minimizer of \reff{1.1}.
\end{thm}
\begin{proof}
Consider the optimization problem:
\be\label{homloc}
\left\{\baray{lll}
\inf & \tilde{f}(\tilde{\bx}) +\epsilon \left(\sum\limits_{i=0}^n x_{i}^{d_0}+\sum\limits_{\ell=1}^p w_{\ell}^{d_0}\right)-f^{(\epsilon)} x_0^d\\
\st  & x_{0} \geq 0,\quad\tilde{g}_{j}(\tilde{\bx}) \geq 0,\quad j\in [m], \\
&\|\tilde{\bx}(\ell)\|^2+w_{\ell}^2=1, \quad\ell \in [p],
\earay \right.
\ee
whose optimal value is clearly $0$. Let $\tilde{\bx}^*\in \Omega$ with $x_0^*>0$.
It follows from \cite[Theorem 3.2]{lasserre2006convergent} that $(\tilde{\bx}^*,\w^*)$ is a minimizer of \reff{homloc} and
\[
\tilde{f}(\tilde{\bx}^*) +\epsilon\left (\sum\limits_{i=0}^n (x^*_{i})^{d_0}+\sum\limits_{\ell=1}^p (w^*_{\ell})^{d_0}\right)-f^{(\epsilon)}(x_0^*)^d=0.
\]
Note that $\tilde{f}(\tilde{\bx}^*)=(x_0^*)^df(\bx^*/x_0^*)$, $\tilde{g}_j(\tilde{\bx}^*)=(x_0^*)^{\deg(g_j)}g(\bx^*/x_0^*)$ $(j\in [m])$. Thus, we have  that $\bx^*/x_0^*\in K$ and  \eqref{sec4:eq1} holds.
\end{proof}


\section{Sparse homogenized Moment-SOS hierarchy without perturbations}\label{sc:shmg2}

For the sparse homogenized hierarchy \reff{spsos}--\reff{spmom} to converge, perturbations are  typically required as illustrated in Remark \ref{rema3.5}. A natural question is whether we can design a sparse homogenized Moment-SOS hierarchy without perturbations while having asymptotic convergence to the optimal value $f_{\min}$ rather than a neighborhood of $f_{\min}$.
In the following, we provide such a hierarchy by introducing a new sparse reformulation of \reff{1.1}.

Suppose that \reff{1.1} admits a csp $(\bx(1),\dots,\bx(p))$. For $i\in[n]$, let $p_i$ denote the frequency of the variable $x_i$ occurring in $\bx(1), \dots, \bx(p)$. Define the set
\be\label{spK1}
\widetilde{K}_s^{2} \coloneqq \left\{(\tilde{\bx},\w) \in \mathbb{R}^{n+1+p}
\left| \baray{l}
x_{0} \geq 0,\quad\tilde{g}_{j}(\tilde{\bx}) \geq 0,\quad j\in [m], \\
\sum\limits_{x_i\in \bx(\ell) } \frac{1}{p_i}x_{i}^2+ \frac{1}{p}x_0^2+w_{\ell}^2=1,\quad\ell \in [p],\\
\|\w\|^2=p-1, \quad \mathbf{1}-\tilde{\bx}^{2}\ge\mathbf{0},\quad\mathbf{1} - \w^{2}\ge\mathbf{0}.\\
\earay \right. \right\}
\ee
Consider the new sparse  reformulation for \reff{1.1} ( $d\coloneqq\deg(f)$):
\begin{equation}\label{hommax1} 
\left\{ \begin{array}{lll}
\sup & \gamma \\
\st &  \tilde{f}(\tilde{\bx})- \gamma x_0^d \geq 0, \quad\forall(\tilde{\bx},\w)\in\widetilde{K}_s^{2}.
\end{array} \right.
\end{equation}

Let $\bu\coloneqq(\tilde{\bx},\w)$. Assume that $(\bu(1),\dots,\bu(q))$ be the list of maximal cliques of some chordal extension of the csp graph associated with POP \reff{hommax1}.
Let
\begin{equation*}
    h_{\ell}'\coloneqq \sum\limits_{x_i\in\bx(\ell)}\frac{1}{p_i}x_{i}^2+ \frac{1}{p}x_0^2+w_{\ell}^2-1\, (\ell\in[p]),~h_{p+1}'\coloneqq\|\w\|^2-p+1.
\end{equation*}
Let $\{J_1,\dots, J_q\}$ be a partition of $[m]$ such that for every $\ell\in [q]$ and every $j\in J_\ell$, $g_j\in\mathbb{R}[\bu(\ell)]$. Moreover, let $\{I_1,\dots, I_q\}$ be a partition of $[p+1]$ such that for every $\ell\in [q]$ and every $j\in I_\ell$, $h_j' \in\mathbb{R}[\bu(\ell)]$.

For the order $k\ge d_{\min}$, the $k$th order sparse SOS relaxation of \reff{hommax1} is 
\begin{equation}\label{spsos1}
\left\{ \begin{array}{lll}
\sup &  \gamma  &\\
\st &  \tilde{f}(\tilde{\bx})  - \gamma x_0^d  \in &\sum\limits_{\ell=1}^q\ideal{\{h'_{j}\}_{j\in I_{\ell}},\bu(\ell)}_{2k}\\ 
& &\quad\quad+\sum\limits_{\ell=1}^q\qmod{\{x_0\}\cup\{1-u_i^2\}_{u_i\in\bu(\ell)}\cup\{\tilde{g}_j\}_{j\in J_\ell},\bu(\ell)}_{2k},\\
\end{array} \right.
\end{equation}

The dual of \reff{spsos1} is the $k$th order sparse moment relaxation:
\be\label{spmom1}
\left\{\begin{array}{lll}
\inf  &\langle \tilde{f}, \by\rangle \\
\st  &\langle x_0^d,\by\rangle=1,\quad M_k[\by, \bu(\ell)] \succeq 0,\quad\ell \in [q],\\
&M_{k-1}[h'_j\by, \bu(\ell)] = 0,\quad j\in I_{\ell},\ell \in [q],\\
&M_{k-1}[x_0\by, \bu(\ell)] \succeq 0,\quad\ell \in [q],\\
& M_{k-\lceil\deg(g_{j})/2\rceil}[\tilde{g}_j \by,\bu(\ell)] \succeq 0, \quad j \in J_\ell, \ell\in [q],\\
&M_{k-1}[(1-u_i^2)\by, \bu(\ell)] = 0,\quad u_i\in\bu(\ell),\ell\in [q]
\end{array}\right.
\ee
Let $\bar{f}_k$, $\bar{f}_k^{\prime}$ be the optimal values of \reff{spsos1} and \reff{spmom1}, respectively.

\subsection{Convergence analysis}\label{sec4-1}
In the following, we show that \reff{1.1} and \reff{hommax1} have the same optimal values.
\begin{theorem}
Suppose that $K$ is closed at infinity. Then the optimal value of \reff{hommax1} is $f_{\min}$.
\end{theorem}
\begin{proof}
Since $K$ is closed at infinity and $f_{\min}>-\infty$, 
we know that $f^{(\infty)}\geq 0$ on $K^{(\infty)}$. Take any $(\tilde{\bx},\w)\in \widetilde{K}_s^{2}$. If $x_0=0$, then we have $g_j^{(\infty)}(\bx)=\tilde{g}_j(\tilde{\bx})\geq 0$ for $j\in [m]$ and thus $\tilde{f}(\tilde{\bx})-f_{\min} x_0^d=f^{(\infty)}(\bx)\geq 0$. 
If $x_0\neq 0$, we have $\bx/x_0\in K$, and 
\[\tilde{f}(\tilde{\bx})- f_{\min} x_0^d=x_0^d(f(\bx/x_0)-f_{\min})\geq 0.\]
It follows that $f_{\min}$ is no greater than the optimal value of \reff{hommax1}.
For the converse, let $\bx^*$ be a minimizer of \reff{1.1}. Let
\[\tilde{\bx}^*=\frac{(1,\bx^*)}{\sqrt{1+\|\bx^*\|^2}},~w^*_{\ell}=\sqrt{1-\sum\limits_{x_i\in \bx(\ell) } \frac{1}{p_i}(x^*_{i})^2-\frac{1}{p}(x_0^*)^2}~(\ell \in [p]).
\]
One can verify $(\tilde{\bx}^*,\w^*)\in \widetilde{K}_s^{2}$. If $\gamma$ is feasible for (\ref{hommax1}), then $
\gamma\leq \tilde{f}(\tilde{\bx}^*)/(x^*_0)^d=f_{\min}$.
\end{proof}

We now establish asymptotic convergence of the sparse homogenized Moment-SOS hierarchy \reff{spsos1}--\reff{spmom1}.
\begin{thm}\label{alt:conve}
Suppose that $K$ is closed at infinity, $f^{(\infty)}$ is positive definite at $\infty$ on $K$, and \reff{1.1} admits a csp $(\bx(1),\dots,\bx(p))$. Then, $\bar{f}_k\rightarrow f_{\min}$ as $k\rightarrow \infty$.
\end{thm}

\begin{proof}
For any $\gamma<f_{\min}$, we show that $\tilde{f}(\tilde{\bx}) -\gamma x_0^d >0$ on $\widetilde{K}_s^{2}$. Take any $(\tilde{\bx},\w)\in \widetilde{K}_s^{2}$. If $x_0=0$, then we must have $\bx\neq \mathbf{0}$. Suppose otherwise $\bx=\mathbf{0}$. Then $w_1^2=w_2^2=\cdots=w_{p}^2=1$,
which contradicts to the constraint $\|\w\|^2-p+1=0$ in the definition of $\widetilde{K}_s^{2}$. Since $f^{(\infty)}$ is positive definite at $\infty$ on $K$, we have
\[\tilde{f}(0,\bx)-\gamma\cdot 0^d=f^{(\infty)}(\bx)> 0.\] 
If $x_0\neq0$, we have $\bx/x_0\in K$ and
\be\tilde{f}(\tilde{\bx})-\gamma(x_0)^d=(x_0)^d(f(\bx/x_0)-\gamma)>0.\ee
Therefore, $\tilde{f}(\tilde{\bx})-\gamma x_0^d >0$ on $\widetilde{K}_s^{2}$ for any $\gamma<f_{\min}$. Moreover, for each $\ell\in[q]$, the quadratic module
\[\qmod{\{x_0\}\cup\{1-u_i^2\}_{u_i\in\bu(\ell)}\cup\{\tilde{g}_j\}_{j\in J_\ell},\bu(\ell)}\] 
is clearly Archimedean in $\mR[\bu(\ell)]$.
It follows from Theorem \ref{spaput} that for any $\gamma<f_{\min}$,
{\small 
\[\tilde{f}(\tilde{\bx})-\gamma x_0^d \in \sum\limits_{\ell=1}^q\left(\ideal{\{h'_{j}\}_{j\in I_{\ell}},\bu(\ell)}+\qmod{\{x_0\}\cup\{1-u_i^2\}_{u_i\in\bu(\ell)}\cup\{\tilde{g}_j\}_{j\in J_\ell},\bu(\ell)}\right).
\]
}
As a result, we obtain $\bar{f}_k\rightarrow f_{\min}$ as $k\rightarrow\infty$.
\end{proof}

\begin{remark}
If the flat truncation condition \eqref{FLT} is satisfied for the $k$th moment relaxation \reff{spmom1}, then $\bar{f}_k=f_{\min}$ and we can extract minimizers of \reff{1.1} via a similar procedure as described in Section \ref{extr:mini}.
\end{remark}

\subsection{Sparse Positivstellens\"atze without perturbations}
In this subsection, we provide new sparse Positivstellens\"atze for positive polynomials on general (possibly unbounded) semialgebraic sets and no perturbations are required.

First, we consider the unconstrained case, i.e., $K=\mR^n$.
\begin{theorem}\label{sparserez1}
Assume that $f>0$ on $\mathbb{R}^n$ and $f^{(\infty)}$ is positive definite.
Then there exist $\sigma_{\ell,i}\in\Sigma[\bu(\ell)]$ for each $u_i\in\bu(\ell),\ell\in [q]$ and $\tau_{\ell,j}\in\mR[\bu(\ell)]$ for each $j\in I_{\ell}, \ell\in [q]$ such that
\begin{equation}
\tilde{f}=\sum\limits_{\ell=1}^q\left(\sum_{u_i\in\bu(\ell)}\sigma_{\ell,i}(1-u_i^2)+\sum_{j\in I_{\ell}}\tau_{\ell,j}h'_{j}\right).
\end{equation}
\end{theorem}

\begin{proof}
Let
\[
S\coloneqq \left\{(\tilde{\bx},\w)\in\mathbb{R}^{n+p+1}
\left| \baray{l}
\sum\limits_{x_i\in \bx(\ell) } \frac{1}{p_i}x_{i}^2+\frac{1}{p}x_0^2+w_{\ell}^2=1,\ell \in [p],\\
\|\w\|^2=p-1,\quad \mathbf{1}-\tilde{\bx}^{2}\ge\mathbf{0},\quad\mathbf{1} - \w^{2}\ge\mathbf{0}.\\
\earay \right. \right\}
\]
We show that $\tilde{f}>0$ on $S$. Take any $(\tilde{\bx},\w)\in S$. 
If $x_0\neq 0$, then we have $\tilde{f}(\tilde{\bx})=x_0^df(\bx/x_0)>0$; If $x_0=0$, then $\bx\neq\mathbf{0}$ and $\tilde{f}(\tilde{\bx})=f^{(\infty)}(\bx)>0$ since $f^{(\infty)}$ is positive definite.
Moreover, for each $\ell\in[q]$, the quadratic module $\sum\limits_{\ell=1}^q\qmod{\{1-u_i^2\}_{u_i\in\bu(\ell)},\bu(\ell)}$
is  Archimedean in $\mR[\bu(\ell)]$. Thus, the conclusion follows from Theorem \ref{spaput}.
\end{proof}

The following theorem addresses the constrained case. As the proof is quite similar to that of Theorem \ref{sparserez1}, we omit it for cleanliness.
\begin{theorem}
Notations follow Section \ref{sec4-1}. Suppose that $K$ is closed at infinity.
If $f>0$ on $K$ and $f$ is positive definite at infinity on $K$, then there exists $\sigma_{\ell}\in\qmod{\{x_0\}\cup\{1-u_i^2\}_{u_i\in\bu(\ell)}\cup\{\tilde{g}_j\}_{j\in J_\ell},\bu(\ell)}$ for each $\ell\in [q]$ and $\tau_{\ell,j}\in\mR[\bu(\ell)]$ for each $j\in I_{\ell}, \ell\in [q]$ such that
\begin{equation}
\tilde{f}=\sum\limits_{\ell=1}^q\left(\sigma_{\ell}+\sum_{j\in I_{\ell}}\tau_{\ell,j}h'_{j}\right).
\end{equation}
\end{theorem}

\subsection{An alternative sparse homogenized Moment-SOS hierarchy without perturbations}\label{alt:spa}
In the description of $\widetilde{K}_s^{2}$, there is a spherical constraint $\|\w\|^2=p-1$ involving all auxiliary variables. When the csp of \reff{1.1} contains a lot of variable cliques, this constraint would lead to a variable clique of big size in the csp of $\widetilde{K}_s^{2}$, which could significantly increase the computational complexity of the hierarchy \reff{spsos1}--\reff{spmom1}. To address this issue, in the following we propose an alternative sparse homogenized Moment-SOS reformulation without perturbations.

Suppose $(\bx(1),\dots,\bx(p))$ is a csp of \reff{1.1}. Recall that for $i\in[n]$, $p_i$ denotes the frequency of the variable $x_i$ occurring in $\bx(1), \dots, \bx(p)$. Define
\be 
\widetilde{K}_s^{3}\coloneqq\left\{(\tilde{\bx},\w)\in \mathbb{R}^{n+p}
\left| \baray{l}
x_{0} \geq 0,\quad\tilde{g}_{j}(\tilde{\bx}) \geq 0, \quad j\in[m], \\
\sum\limits_{x_i\in \bx(1)} \frac{1}{p_i}x_{i}^2+ \frac{1}{p}x_0^2=w_1^2, \\
\sum\limits_{x_i\in \bx(2)} \frac{1}{p_i}x_{i}^2+ \frac{1}{p}x_0^2+w_1^2=w_2^2,\\
\quad\quad\quad\quad\vdots\\
\sum\limits_{x_i\in \bx(p)} \frac{1}{p_i}x_{i}^2+ \frac{1}{p}x_0^2+w_{p-1}^2=1,\\
\mathbf{1} - \tilde{\bx}^{2}\ge\mathbf{0},\quad\mathbf{1} - \w^{2}\ge\mathbf{0},\\
\earay \right. \right\}
\ee
where $\w\coloneqq(w_1,\ldots,w_{p-1})$.
Consider the following sparse homogenized reformulation for \reff{1.1}:
\begin{equation}\label{hommax2} 
\left\{ \begin{array}{lll}
\sup & \gamma \\
\st &  \tilde{f}(\tilde{\bx})- \gamma x_0^d \geq 0, \quad\forall(\tilde{\bx},\w)\in \widetilde{K}_s^{3}.
\end{array} \right.
\end{equation}
Similarly, one can verify that \reff{hommax2} has the same optimal value with \reff{1.1}. Furthermore, the asymptotic convergence of the corresponding sparse homogenized Moment-SOS hierarchy for \reff{hommax2} as well as related sparse Positivstellens\"atze can be established using similar arguments as in the previous subsections. 

\vspace{1em}
So far, we have discussed how to exploit correlative sparsity for homogenized polynomial optimization but do not touch term sparsity. Actually, correlative sparsity and term sparsity can be exploited simultaneously to gain further reductions on the size of SDP relaxations arising from Moment-SOS hierarchies. We refer the reader to \cite{magron2023sparse,cstssos} for details.

\section{Numerical examples}\label{num:ex}
In this section, we present numerical results on solving POPs with three sparse homogenized Moment-SOS hierarchies.
All numerical experiments are performed on a desktop computer with Intel(R) Core(TM) i9-10900 CPU@2.80GHz and 64G RAM. To model the homogenized hierarchies, we use the Julia package {\tt TSSOS}\footnote{{\tt TSSOS} is freely available at \href{https://github.com/wangjie212/TSSOS}{https://github.com/wangjie212/TSSOS}.} \cite{magron2021tssos}, relying on {\tt Mosek} 10.0 \cite{mosek} as an SDP backend with default settings. Unless
otherwise specified, we set $\epsilon=10^{-4}$ for relaxations \reff{spsos}--\reff{spmom}. We do not implement and compare with the approach proposed in \cite{mai2023sparse} since it is limited to problems of modest size. Notations are listed in Table \ref{not}.

\begin{table}[htbp]\label{not}
\caption{Notation}
\centering
\begin{tabular}{c||c}
$n$&number of variables\\
\hline
$k$&relaxation order\\
\hline
opt&optimum\\
\hline
time&running time in seconds\\
\hline
SSOS&the sparse SOS relaxation \reff{spasos}\\
\hline
HSOS&the dense homogenized SOS relaxation\\
\hline
HSSOS1&the sparse homogenized SOS relaxation \reff{spsos}\\
\hline
HSSOS2&the sparse homogenized SOS relaxation \reff{spsos1}\\
\hline
HSSOS3&the alternative sparse homogenized SOS relaxation in Section \ref{alt:spa}\\
\hline
{\bf bold font}&global optimality being certified\\
\hline
*&indicating unknown termination status\\
\hline
$**$&infeasible SDP\\
\hline
-&returning an out of memory error\\
\end{tabular}
\end{table}

\subsection{Unconstrained polynomial optimization}
\begin{example}
Let $\bx(1)=\{x_1,x_2,x_3\}$ and $\bx(2)=\{x_1,x_2,x_4\}$. Consider   POP \reff{1.1}  with csp $(\bx(1),\bx(2))$, where
$$f=f_1+f_2,\quad f_1=x_3^2(x_1^2+x_1^4 x_2^2+ x_3^4-3 x_1^2 x_2^2)+x_2^8,\quad f_2=x_1^2 x_2^2 x_4^2.$$
The polynomial $f_1$ is the dehomogenized Delzell's polynomial, which is nonnegative but not an SOS \cite{reznick2000some}. This example is a variation of Example 1 in \cite{mai2023sparse}. As shown in \cite{mai2023sparse}, $f$ is nonnegative and $f\notin \Sigma\left[\bx(1)\right]+\Sigma\left[\bx(2)\right]$. By solving \eqref{spsos} with $\epsilon=0$ and $k=5$, we obtain $f_5\approx-1.6\times10^{-7}$, which confirms $f_5=f_{\min}=0$ (up to numerical round-off errors). 
\end{example}

\begin{example} \label{un:ex3}
Let
$$\bx(1)=\{x_1,x_2,x_3,x_4\}, \quad \bx(2)=\{x_4,x_5,x_6,x_7\}, \quad \bx(3)=\{x_7,x_8,x_9,x_{10}\}.$$
Consider POP \reff{1.1} with  csp $(\bx(1), \bx(2), \bx(3))$, where $f=f_1+f_2+f_3$ $(x_0\coloneqq1)$
\begin{align*}
f_1&=\sum_{i=1}^4 x_i^4+\sum_{i=0}^4 \prod_{j \neq i}\left(x_i-x_j\right), \\
f_2&=\sum_{i=4}^7 x_i^4+\sum_{i=0,4,\dots,7} \prod_{j \neq i}\left(x_i-x_j\right), \\
f_3&=\sum_{i=7}^{10} x_i^4+\sum_{i=0,7,\dots,10} \prod_{j \neq i}\left(x_i-x_j\right).
\end{align*}
Here we set $\epsilon=0$ for HSSOS1. The numerical results for this problem are presented in Table \ref{tabun:ex3}. From the table, we can draw the following conclusions: (1) Without homogenization, the sparse hierarchy converges slowly; (2) By exploiting sparsity, we gain a significant speed-up especially when the relaxation order is high; (3) All three sparse homogenized Moment-SOS hierarchies achieve the optimum $f_{\min}\approx0.6927$ at the third order relaxation.	 

\begin{table}[htbp]\label{tabun:ex3}
\caption{Results of Example \ref{un:ex3}}
\renewcommand\arraystretch{1.2}
\centering
\resizebox{\linewidth}{!}{
\begin{tabular}{c|c|c|c|c|c|c|c|c|c|c}
\multirow{2}{*}{$k$}& \multicolumn{2}{c|}{SSOS}& \multicolumn{2}{c|}{HSOS}& \multicolumn{2}{c|}{HSSOS1}& \multicolumn{2}{c|}{HSSOS2}& \multicolumn{2}{c}{HSSOS3}\\
\cline{2-11}
& opt &time & opt &time &opt &time&opt &time&opt &time\\
\hline
2 &0.5497&0.02& 0.5497 & 0.05&0.5497&0.03&0.5497&0.02&0.5497&0.03\\
\hline
3 &0.5497&0.21& {\bf 0.6927} & 13.3&{\bf 0.6927}&0.37&{\bf 0.6927}&0.15&{\bf 0.6927}&0.20\\
\hline
4 &0.5864*&0.73& {\bf 0.6927}& 683 &{\bf 0.6927}&3.27&{\bf 0.6927}&1.38&{\bf 0.6927}&1.77\\
\end{tabular}}
\end{table}

\end{example}

\begin{example} \label{un:ex4}
Let
\begin{align*}
\bx(1)&=\{x_1,x_2,x_3,x_4,x_5\}, &\bx(2)&=\{x_1,x_2,x_6,x_7,x_8\}, \\
\bx(3)&=\{x_1,x_2,x_9,x_{10},x_{11}\},&\bx(4)&=\{x_1,x_2,x_{12},x_{13},x_{14}\},\\ \bx(5)&=\{x_1,x_2,x_{15},x_{16},x_{17}\}, &\bx(6)&=\{x_1,x_2,x_{18},x_{19},x_{20}\}.
\end{align*}
Consider POP \reff{1.1} with csp $(\bx(1), \bx(2), \bx(3), \bx(4), \bx(5), \bx(6))$, where $f=\sum_{i=1}^6f_i$, and for $i=1,\dots,6$,
\begin{align*}
f_i=&\,x_{1}^{2}\left(x_{1}-1\right)^{2}+x_{2}^{2}\left(x_{2}-1\right)^{2}+
x_{3i}^{2}\left(x_{3i}-1\right)^{2}   +2 x_{1} x_{2} x_{3i}\left(x_{1}+x_{2}+x_{3i}-2 \right)\\	
&  +\frac{1}{4}(\left(x_{1}-1\right)^{2}+\left(x_{2}-1\right)^{2}+\left(x_{3i}-1\right)^{2}+\left(x_{3i+1}-1\right)^{2})+(x_{3i+1}x_{3i+2}-1)^2.	
\end{align*}
Here we set $\epsilon=0$ for HSSOS1. The numerical results for this problem are presented in Table \ref{tabun:ex4}. From the table, we can draw the following conclusions: (1) Without homogenization, the sparse hierarchy converges slowly; (2) By exploiting sparsity, we gain a significant speed-up and reach relaxations of higher orders; (3) HSSOS1 achieves the optimum at $k=4$ and HSSOS3 achieves the optimum at $k=3$, whereas HSSOS2 gives wrong answers due to numerical issues when $k\ge3$.

\begin{table}[htbp]\label{tabun:ex4}
\caption{Results of Example \ref{un:ex4}}
\renewcommand\arraystretch{1.2}
\centering
\resizebox{\linewidth}{!}{
\begin{tabular}{c|c|c|c|c|c|c|c|c|c|c}
\multirow{2}{*}{$k$}& \multicolumn{2}{c|}{SSOS}& \multicolumn{2}{c|}{HSOS}&\multicolumn{2}{c|}{HSSOS1}&\multicolumn{2}{c|}{HSSOS2}&\multicolumn{2}{c}{HSSOS3}\\
\cline{2-11}
& opt &time & opt &time &opt &time&opt &time&opt &time\\
\hline
2 &1.1804&0.01& 1.1804 &0.54&1.1804&0.04&1.1804&0.09&1.1804&0.11\\
\hline
3 &1.1804&0.07& - & - &1.1895&0.34&1.1969*&1.18&{\bf 1.1900}&0.96\\
\hline
4 &1.1809&0.40& - & - &{\bf 1.1900}&1.48&1.4871*&17.4&1.1901&5.94\\
\end{tabular}}
\end{table}

\end{example}

\subsection{Constrained polynomial optimziation}
\begin{example} \label{con:ex1}
Let $$\bx(1)=\{x_1,x_2\}, \quad \bx(2)=\{x_2,x_3\}, \quad \bx(3)=\{x_2,x_4,x_5\}.$$
Consider  POP \reff{1.1} with csp $(\bx(1), \bx(2), \bx(3))$:
\be
\left\{ \begin{array}{lll}
\inf & x_1^2+3x_2^2-2x_2x_3^2+x_3^4-x_2(x_4^2+ x_5^2)\\
\st  & x_1^2-2x_1x_2-1\geq 0,\, x_1^2+2x_1x_2-1\geq 0,\\
& x_2^2-1 \geq 0, \,x_2-x_6^2- x_7^2\geq 0.
\end{array}\right.
\ee
For this problem, the optimal value is $4+2\sqrt{2}\approx6.8284$. The numerical results of this problem are presented in Table \ref{tabcon:ex1}. From the table, we can draw the following conclusions: (1) Without homogenization, the sparse hierarchy converges slowly; (2) HSOS, HSSOS2, and HSSOS3 all achieve the optimum at $k=4$, while HSSOS1 converges more slowly.

\begin{table}[htbp]\label{tabcon:ex1}
\renewcommand\arraystretch{1.2}
\centering 
\caption{Results of Example \ref{con:ex1}}
\resizebox{\linewidth}{!}{
\begin{tabular}{c|c|c|c|c|c|c|c|c|c|c}
\multirow{2}{*}{$k$}& \multicolumn{2}{c|}{SSOS}& \multicolumn{2}{c|}{HSOS}& \multicolumn{2}{c|}{HSSOS1}& \multicolumn{2}{c|}{HSSOS2}&\multicolumn{2}{c}{HSSOS3}\\
\cline{2-11} 
& opt &time & opt &time &opt &time  &opt &time&opt &time\\
\hline
2 &2.0000&0.05&2.0000& 0.01&2.0048&0.01&2.0000&0.01&2.0000&0.01\\
\hline
3 & 2.0343* &0.05 & 5.1310* & 0.06&2.8286&0.05 &4.9449*&0.09&4.9678&0.09\\
\hline
4 & 2.1950*& 0.07 &{\bf 6.8284}& 0.19&4.1178&0.19&{\bf 6.8284}&0.15&{\bf 6.8284}&0.18\\
\end{tabular}}
\end{table}

\end{example}

\begin{example}\label{con:ex2}
Let $$\bx(1)=\{x_1,x_2,x_3,x_7\}, \quad \bx(2)=\{x_4,x_5,x_6,x_7\}.$$
Consider POP \reff{1.1} with csp $(\bx(1), \bx(2))$:
\begin{equation*}
\left\{ \begin{array}{lll}
\inf & f_1+f_2 \\
\st  & x_1-x_2x_3\ge0,-x_2+x_3^2\ge0,1-x_4^2-x_5^2-x_6^2\ge0,
\end{array}\right.
\end{equation*}
where
\begin{align*}
f_1&=x_1^4x_2^2+x_2^4x_3^2+x_3^4x_1^2-3(x_1x_2x_3)^2+x_2^2 + x_7^2(x_1^2+x_2^2+x_3^2),\\
f_2&=x^2_4x^2_5(10-x^2_6)+x_7^2(x^2_4+2x^2_5+3x^2_6).
\end{align*}
For this problem, the optimal value is $0$ \cite{qu2022correlative}. The numerical results are presented in Table \ref{tabcon:ex2}. From the table, we make the following observations: (1) Without homogenization, the sparse hierarchy either yields infeasible SDPs or gives very looser bounds; (2) By exploiting sparsity, we gain some speed-up; (3) HSOS achieves the optimum at $k=4$, and both HSSOS2 and HSSOS3 achieves the optimum at $k=5$, while HSSOS1 converges to a near neighbourhood of $f_{\min}$ at $k=4$.

\begin{table}[htbp]\label{tabcon:ex2}
\caption{Results of Example \ref{con:ex2}}
\renewcommand\arraystretch{1.2}
\centering
\resizebox{\linewidth}{!}{
\begin{tabular}{c|c|c|c|c|c|c|c|c|c|c}
\multirow{2}{*}{$k$}& \multicolumn{2}{c|}{SSOS}& \multicolumn{2}{c|}{HSOS}&\multicolumn{2}{c|}{HSSOS1}&\multicolumn{2}{c|}{HSSOS2}&\multicolumn{2}{c}{HSSOS3}\\
\cline{2-11}
& opt &time & opt &time &opt &time&opt &time&opt &time\\
\hline
3&$**$&0.04&-4532&0.28&-1756*&0.16&-1065*&0.24&-1106*&0.20\\
\hline
4&$**$&0.19&{\bf -1.6e-8}&2.71&0.0001&0.82&-0.0002&1.37&-0.0002&1.77\\
\hline
5&-4.0e5&0.89&{\bf -9.8e-9}&33.4&0.0001&5.33&{\bf 1.1e-7}&6.98&{\bf 1.4e-7}&6.38\\
\end{tabular}}
\end{table}

\end{example}

\begin{example}	\label{un:ex5}
For an integer $p\ge2$, let
$$\bx(i)=\{x_{8i-7},x_{8i-6},\dots,x_{8i+2}\}, \quad i\in[p].$$
For $i\in[p]$, let
\begin{align*}
f_i=&\,\left(\sum\limits_{j=1}^{10} \left(x_j^{(i)}\right)^2+1\right)^2-4\left(\left(x_1^{(i)} x_2^{(i)}\right)^2+\cdots+\left(x_4^{(i)} x_5^{(i)}\right)^2+\left(x_5^{(i)} x_1^{(i)}\right)^2\right)\\	
& -4\left(\left(x_6^{(i)} x_7^{(i)}\right)^2+\cdots+\left(x_9^{(i)} x_{10}^{(i)}\right)^2+\left(x_{10}^{(i)} x_6^{(i)}\right)^2\right) +\frac{1}{5}\sum\limits_{j=1}^{10} \left(x_j^{(i)}\right)^4.
\end{align*}
Consider POP \reff{1.1} with csp $(\bx(1),\bx(2),\ldots,\bx(p))$:
\begin{equation*}
\left\{ \begin{array}{lll}
\inf & \sum_{i=1}^pf_i \\
\st  & \|\bx(i)^{2}\|^2-1\ge0,\quad i=1,\ldots,p.
\end{array}\right.
\end{equation*}
We solve the fourth order relaxations for different $p$.
The numerical results for this problem are presented in Table \ref{tabun:ex5}. From the table, we can draw the following conclusions: (1) Without homogenization, the sparse relaxation yields very looser bounds; (2) For $p=2,3$, HSOS achieves the optimum while for $p\ge4$, HSOS runs out of memory; (3) By exploiting sparsity, we improve the scalability of the homogenization approach and still obtain good bounds.

\begin{table}[htbp]\label{tabun:ex5}
\caption{Results of Example \ref{un:ex5}}
\renewcommand\arraystretch{1.2}
\centering
\resizebox{\linewidth}{!}{
\begin{tabular}{c|c|c|c|c|c|c|c|c|c|c}
\multirow{2}{*}{$p$}& \multicolumn{2}{c|}{SSOS}& \multicolumn{2}{c|}{HSOS}&\multicolumn{2}{c|}{HSSOS1}&\multicolumn{2}{c|}{HSSOS2}&\multicolumn{2}{c}{HSSOS3}\\
\cline{2-11}
& opt &time & opt &time &opt &time&opt &time&opt &time\\
\hline
2 &-11053*&2.72&{\bf 6.1488}&156&6.0984&14.6&{\bf 6.1488}&15.5&{\bf 6.1488}&12.8\\
\hline
3 & -18999* &4.14  & {\bf 9.2232} & 2763 &9.1475&20.1&9.2227&20.6&9.2228&31.3\\
\hline
4 & -26984* &5.14  & - & - &12.196&29.4&12.294&30.4&12.295&55.1\\
\hline
5 & -31198* &6.54  & - & - &15.246&39.0&15.365&39.4&15.364&69.4\\
\hline
10 & -80847*& 12.8& - & - &30.491&101&30.543&122&30.504&170\\
\end{tabular}}
\end{table}

\end{example}

\begin{example}\label{con:ex3}
We generate random instances of quadratic  optimization on unbounded sets as follows. For $n\in\{20,40,100,200,400,800,2000\}$, let $p=\lceil \frac{n}{3}\rceil$. Let 
{\small
\[
\bx(1)=\{x_1,x_2,x_3\},\,\bx(i)=\{x_{3(i-1)},\ldots,x_{3i}\}~(i=2,\ldots,p-1),\,\bx(p)=\{x_{3(p-1)},\ldots,x_{n}\}.
\]
}
Let $A_1\in\R^{3\times3}$, $b_1\in\R^{3}$, $A_i\in\R^{4\times4},b_i\in\R^{4}~(i=2,\ldots,p-1)$,  $b_{p}\in\R^{n-3p+4}$, $A_{p}\in\R^{(n-3p+4)\times(n-3p+4)}$ be randomly generated with entries being uniformly taken from $[0,1]$. For $i=1,\ldots,p$, let $f_i=\|A_i\bx(i)^{2}\|^2+b_i^{\intercal}\bx(i)^{2}$. 
Consider the POP with csp $(\bx(1),\bx(2),\ldots,\bx(p))$:
\begin{equation*}
\left\{ \begin{array}{lll}
\inf & \sum_{i=1}^pf_i \\
\st  & \|\bx(i)^{2}\|^2-1\ge0,\quad i=1,\ldots,p.
\end{array} \right.
\end{equation*}
We solve the fourth order relaxations for different $n$.
The numerical results for this problem are presented in Table \ref{tabcon:ex3}. From the table, we can draw the following conclusions: (1) Without homogenization, the sparse relaxation yields much looser bounds; (2) By exploiting sparsity, we gain a significant speed-up; (3) Both HSOS and HSOSS2 do not scale well with the problem size (HSOS runs out of memory when $n\ge40$ and HSOSS2 runs out of memory when $n\ge100$); (4) Both HSSOS1 and HSOSS3 scale well with the problem size (up to $n=2000$).

\begin{table}[htbp]\label{tabcon:ex3}
\caption{Results of Example \ref{con:ex3}}
\renewcommand\arraystretch{1.2}
\centering
\resizebox{\linewidth}{!}{
\begin{tabular}{c|c|c|c|c|c|c|c|c|c|c}
\multirow{2}{*}{$n$}& \multicolumn{2}{c|}{SSOS}& \multicolumn{2}{c|}{HSOS}&\multicolumn{2}{c|}{HSSOS1}&\multicolumn{2}{c|}{HSSOS2}&\multicolumn{2}{c}{HSSOS3}\\
\cline{2-11}
& opt &time & opt &time &opt &time&opt &time&opt &time\\
\hline
20&5.5065&0.33&{\bf 8.8328}&240&8.4216&0.93&{\bf 8.8328}&1.21&{\bf 8.8328}&1.52\\ 
\hline
40&11.813&0.35&-&-&17.481&1.51&18.059&29.7&17.856&3.59\\ 
\hline
100&27.976&1.29&-&-&42.273&6.94&-&-&41.336&16.3\\
\hline
200&60.178&2.62&-&-&87.726&19.7&-&-&82.240&52.9\\
\hline
400&111.35&6.52&-&-&164.06&55.4&-&-&146.66*&190\\
\hline
800&228.42&18.7&-&-&337.01&229&-&-&296.70*&702\\
\hline
2000&577.88&88.0&-&-&854.14*&1736&-&-&768.31*&6424\\
\end{tabular}}
\end{table}

\end{example}


\section{Applications to trajectory optimization}
\label{tra:opt}
Trajectory optimization plays an essential role in the fields of robotics and control~\cite{betts1998jgcd-traopt-survey}.~\cite{teng2023arxiv-geometricmotionplanning-liegroup} demonstrates applying the sparse Moment-SOS hierarchy to specific trajectory optimization problems with compact feasible sets tends to yield tight solutions. However, assuming pre-defined bounds over all physical quantities can be unrealistic. For instance, it is particularly hard to bound the generalized momentum of highly nonlinear systems or the contact forces in contact-rich scenarios~\cite{aydinoglu2021tro-stabilization-complementary}. In such contexts, the ability to relax the compactness assumption while still achieving tight solutions is desirable. 
In this section, we explore two trajectory optimization problems with unbounded feasible sets: (1) block-moving with minimum work using direct collocation; (2) optimal control of Van der Pol oscillator with direct multiple shooting. We compare the performance of SSOS, HSSOS1, and HSSOS3 (noting that HSOS and HSSOS2 do not scale with large clique numbers).

\subsection{Block-moving with minimum work}

The continuous time version of block-moving with minimum work is shown as follows~\cite{kelly2017siam-intro-directcollocation}:
\begin{equation}
    \label{eq:sec6-bmw-ct}
    \left\{ \begin{array}{rl}
    \min\limits_{u(\tau), x_1(\tau), x_2(\tau)} &\int_{\tau=0}^{1} |u(\tau) x_2(\tau)| \,\rm{d}\tau \\
    \text{s.t. } \quad\,\,\,\,\,& \frac{d}{dt} \begin{bmatrix}
        x_1 \\ x_2
    \end{bmatrix} = f(x, u) = \begin{bmatrix}
        x_2 \\ u
    \end{bmatrix}, \\
    & x_1(0) = 0, x_2(0) = 0, x_1(1) = 1, x_2(1) = 0 ,
    \end{array} \right.
\end{equation}
where $x_1$ and $x_2$ are the block's position and velocity respectively. Starting from the origin in state space $\bx(0) = [0, 0]^{\intercal}$, our goal is to push the block to a terminal state $\bx(1) = [1, 0]^{\intercal}$ at time $t=1$, while minimizing the work done. To achieve this, slack variables are introduced, and direct collocation is applied to discretize~\eqref{eq:sec6-bmw-ct}. This process results in the following POP:
\begin{equation}
    \label{eq:sec6-bmw-pop}
    \left\{ \begin{array}{rl}
    \min\limits_{\substack{
        u_k, \ k = 0, \dots, N \\
        x_{k, 1}, x_{k, 2}, \ k = 0, \dots, N \\
        s_{k, 1}, s_{k, 2}, \ k = 0, \dots, N 
    }} &\sum_{k=0}^N (s_{k, 1} + s_{k, 2}) \cdot h \\
    \text{s.t. } \quad\quad\,\,& s_{k, 1} \ge 0, s_{k, 2} \ge 0, \quad k = 0, \dots, N,  \\
    & s_{k, 1} - s_{k, 2} = u_k \cdot x_{k, 2}, \quad k = 0, \dots, N,  \\
    & \dot{x}_{k, c} = f(x_{k, c}, u_{k, c}), \quad k = 0, \dots, N-1, \\
    & x_{0, 1} = 0, x_{0, 2} = 0, x_{N, 1} = 1, x_{N, 2} = 0, 
    \end{array} \right.
\end{equation}
where $N$ is the total time steps and $h$ is the time step. Since the terminal time is fixed as $1$, $N \cdot h = 1$ should hold. Here $\dot{x}_{k, c}$, $x_{k, c}$, and $u_{k, c}$ in~\eqref{eq:sec6-bmw-pop} stems from the collocation constraints.


It should be noted that~\eqref{eq:sec6-bmw-pop} is a non-convex problem due to the inclusion of quadratic equality constraints.~\eqref{eq:sec6-bmw-pop} exhibits a chain-like csp. Specifically, if we consider the $k$th clique as $(u_{k-1}, x_{k-1}, s_{k-1}, u_k, x_k, s_k), k=1,\dots,N$, the RIP is satisfied due to the Markov property. Setting relaxation order $k$ to $2$, we test the three algorithms' performance on multiple values of $N$ and $u_{\max}$. For HSSOS1, the perturbation parameter $\epsilon$ is set to $10^{-4}$. The results are shown in Table~\ref{tab:bmw}. Note that we also reported the sub-optimality gap $\eta$ between SDP's solution and the solution refined by nonlinear programming solvers (here we use MATLAB's fmincon). Denote SDP's optimal value as $f_{\text{lower}}$ and fmincon's local minimum as $f_{\text{upper}}$. Then $\eta$ is defined as
\begin{equation}
    \label{eq:suboptimality-gap}
    \eta = \frac{
        |f_{\text{upper}} - f_{\text{lower}}|
    }{
        1 + |f_{\text{upper}}| + |f_{\text{lower}}|
    }.
\end{equation}

This gap is shown in logarithmic form, i.e., $\log_{10}\eta$. 
From Table~\ref{tab:bmw}, we see that HSSOS3 achieves tight solutions in all parameter settings, with the sub-optimiality gap always lower than $10^{-4}$. However, both SSOS and HSSOS1 suffer from numerical issues. Further trajectory visualizations for $u(t)$ are given in Figure~\ref{fig:bmw}.  


\begin{table}[htbp]
    \label{tab:bmw}
    \caption{Results of the block-moving example}
    \centering
    \renewcommand\arraystretch{1.2}
    \resizebox{\linewidth}{!}{
    \begin{tabular}{c|c|c|c|c|c|c|c|c|c|c}
         \multirow{2}{*}{$N$} & \multirow{2}{*}{$u_{\max}$} & \multicolumn{3}{c|}{SSOS} & \multicolumn{3}{c|}{HSSOS1} & \multicolumn{3}{c}{HSSOS3}  \\
         \cline{3-11}
         & & opt & time & gap & opt & time & gap & opt & time & gap \\
         \hline 
         \multirow{6}{*}{$10$} &  10&  $1.1164$&  2.90&  $-2.62$&  $0.5044^*$&  $25.9$&  $-0.43$&  $1.1111$&  $10.0$&  $-12.0$\\
         \cline{2-11}
         &  12&  $0.9660$&  3.26&  $-2.74$&  $0.4391^*$&  $26.5$&  $-0.43$&  $0.9625$&  $12.5$& $-10.5$\\
         \cline{2-11}
         &  14&  $0.8100^*$&  6.30&  $-2.01$&  $0.3774^*$&  $24.2$&  $-0.45$&  $0.7944$&  $10.8$& $-10.2$\\
         \cline{2-11}
         &  16&  $0.6370^*$&  4.72&  $-1.61$&  $0.3101^*$&  $26.4$&  $-0.49$&  $0.6069$&  $13.9$& $-12.9$\\
         \cline{2-11}
         &  18&  $0.6448^*$&  5.43&  $-0.63$&  $0.2565^*$&  $23.4$&  $-0.55$&  $0.4000$&  $10.6$& $-7.89$\\
         \cline{2-11}
         &  20&  $0.3651^*$&  3.41&  $-0.45$&  $0.1370$&  $28.4$&  $-0.92$&  $0.1741$&  $11.9$& $-8.17$\\
         \hline 
         \multirow{6}{*}{$20$} &  10&  $1.2301^*$&  7.35&  $-2.76$&  $0.6873^*$&  $51.0$&  $-0.55$&  $1.2266$&  $29.3$&  $-10.9$\\
         \cline{2-11}
         &  12&  $1.1725^*$&  10.4&  $-1.97$&  $0.6663^*$&  $59.6$&  $-0.57$&  $1.1476$&  $27.7$& $-11.1$\\
         \cline{2-11}
         &  14&  $1.1410^*$&  12.7&  $-1.46$&  $0.6600^*$&  $84.2$&  $-0.63$&  $1.0646$&  $34.3$& $-7.40$\\
         \cline{2-11}
         &  16&  $1.0924^*$&  9.23&  $-1.40$&  $0.6181^*$&  $56.3$&  $-0.62$&  $1.0096$&  $26.6$& $-8.56$\\
         \cline{2-11}
         &  18&  $1.0854^*$&  8.31&  $-1.21$&  $0.5918^*$&  $59.6$&  $-0.63$&  $0.9591$&  $27.1$& $-8.04$\\
         \cline{2-11}
         &  20&  $1.0278^*$&  8.70&  $-1.19$&  $0.5680^*$&  $74.7$&  $-0.64$&  $0.9036$&  $29.7$& $-8.66$\\
         \hline 
         \multirow{6}{*}{$30$} &  10&  $1.2724^*$&  9.50&  $-2.02$&  $0.8184^*$&  $73.5$&  $-0.68$&  $1.2483$&  $42.7$&  $-9.22$\\
         \cline{2-11}
         &  12&  $1.2107^*$&  11.5&  $-1.92$&  $0.8024^*$&  $23.2$&  $-0.72$&  $1.1822$&  $63.0$& $-8.10$\\
         \cline{2-11}
         &  14&  $1.1985^*$&  10.1&  $-1.52$&  $0.7787^*$&  $23.4$&  $-0.74$&  $1.1294$&  $44.3$& $-7.91$\\
         \cline{2-11}
         &  16&  $1.1804^*$&  10.1&  $-1.41$&  $0.7649^*$&  $23.5$&  $-0.75$&  $1.0923$&  $47.4$& $-8.27$\\
         \cline{2-11}
         &  18&  $1.1718^*$&  8.38&  $-1.27$&  $0.7436^*$&  $22.6$&  $-0.76$&  $1.0532$&  $54.9$& $-7.06$\\
         \cline{2-11}
         &  20&  $1.2133^*$&  6.74&  $-1.04$&  $0.7331^*$&  $27.1$&  $-0.80$&  $1.0119$&  $67.7$& $-4.07$\\

    \end{tabular}}
\end{table}

\begin{figure}[htbp]
    \begin{minipage}{\textwidth}
        \centering
        $N = 10$
        \vspace{0.2cm}
        \begin{tabular}{ccc}
            \begin{minipage}{0.31\textwidth}
                \vspace{0.2cm}
                \centering
                \includegraphics[width=\columnwidth]{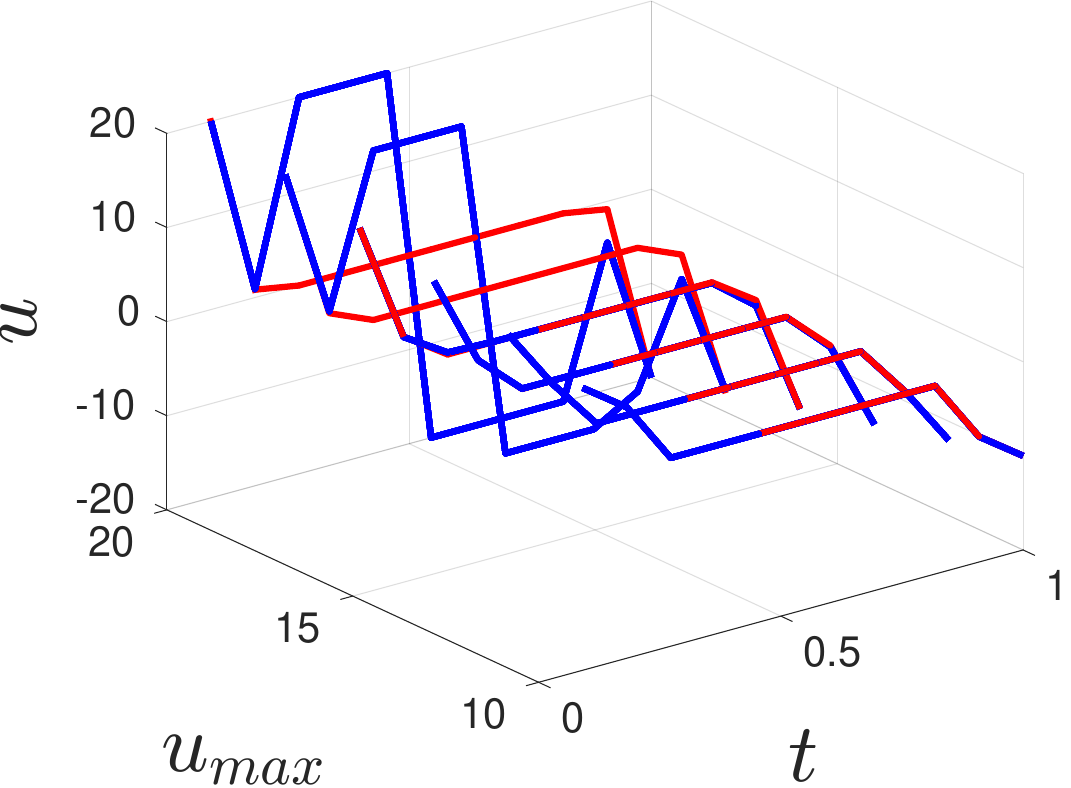}
            \end{minipage}
            
            \begin{minipage}{0.31\textwidth}
                \vspace{0.2cm}
                \centering
                \includegraphics[width=\columnwidth]{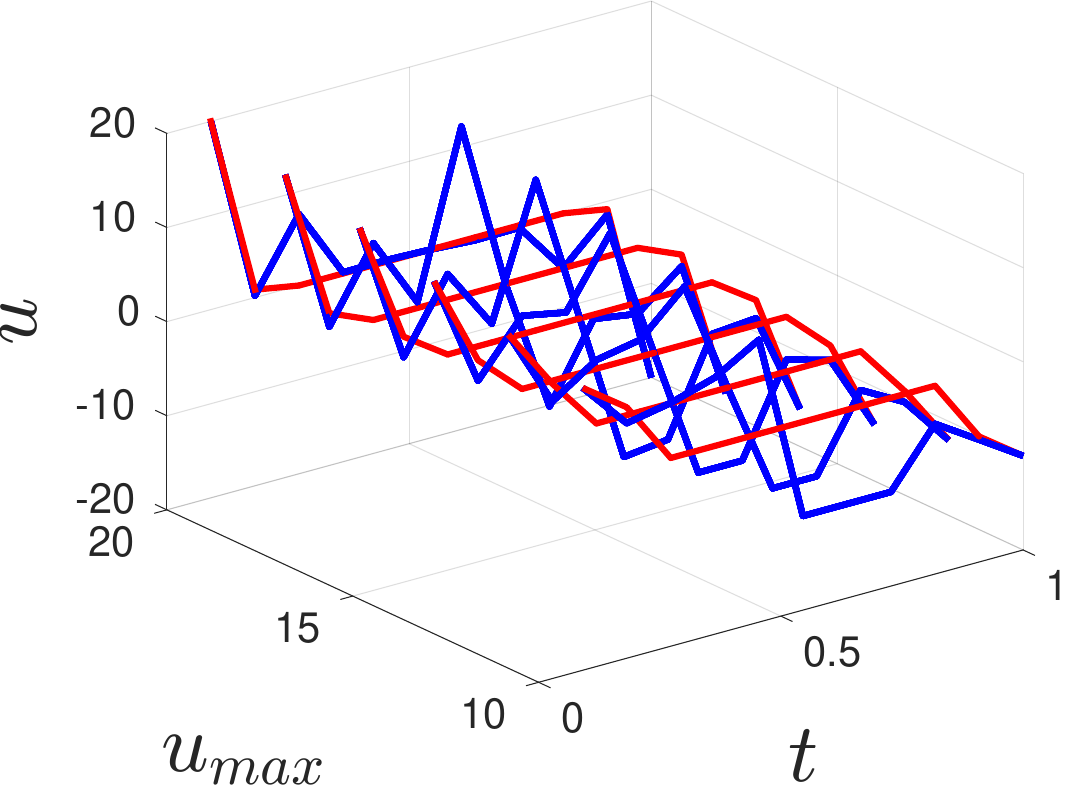}
            \end{minipage} 
            
            \begin{minipage}{0.31\textwidth}
                \vspace{0.2cm}
                \centering
                \includegraphics[width=\columnwidth]{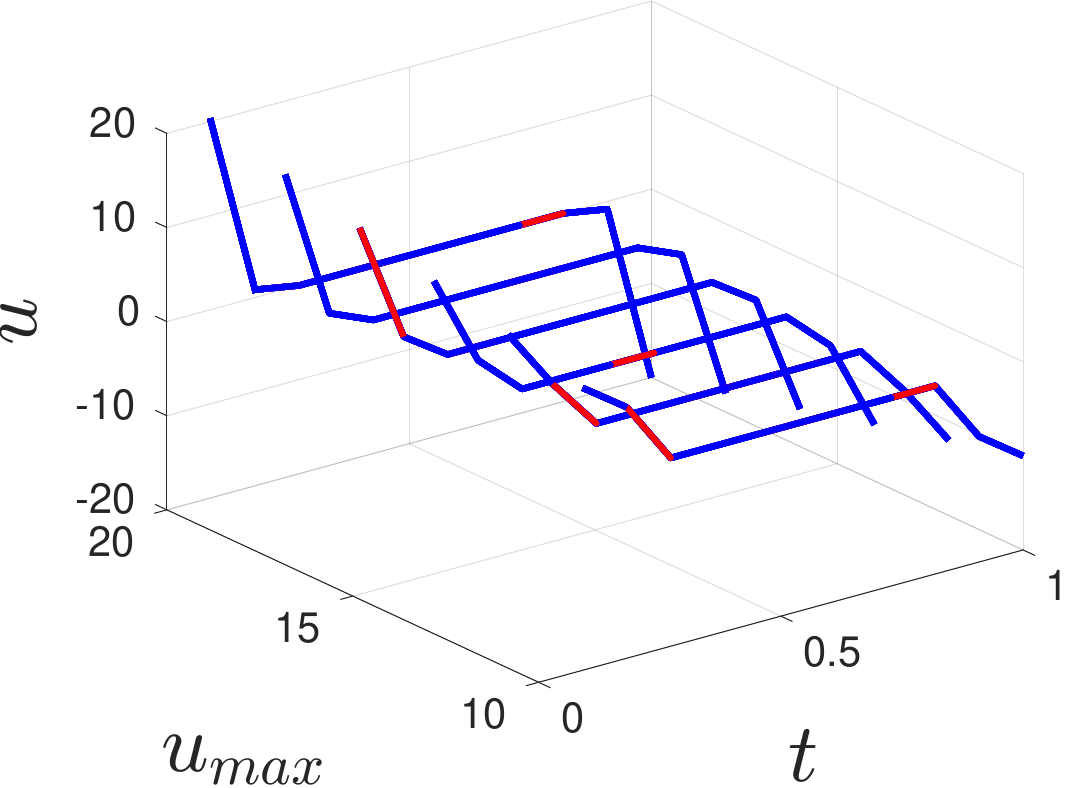}
            \end{minipage} 
        \end{tabular}
    \end{minipage}

    \begin{minipage}{\textwidth}
        \centering
        $N = 20$
        \vspace{0.2cm}
        \begin{tabular}{ccc}
            \begin{minipage}{0.31\textwidth}
                \vspace{0.2cm}
                \centering
                \includegraphics[width=\columnwidth]{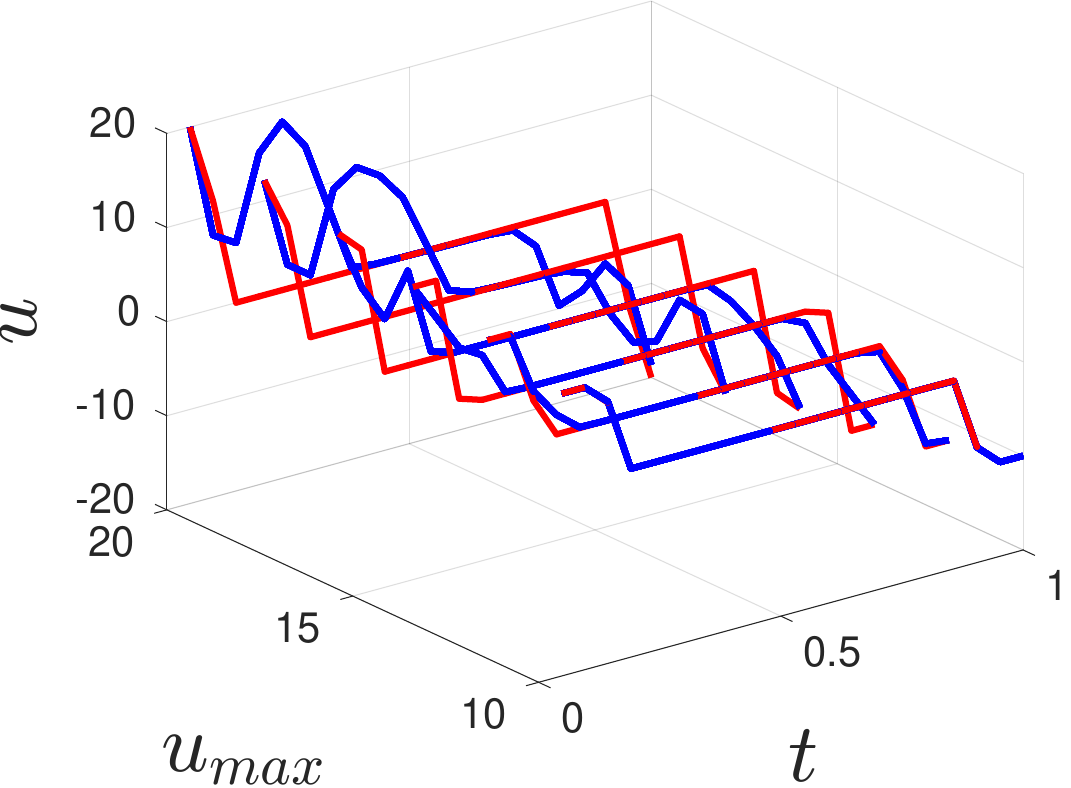}
            \end{minipage}
            
            \begin{minipage}{0.31\textwidth}
                \vspace{0.2cm}
                \centering
                \includegraphics[width=\columnwidth]{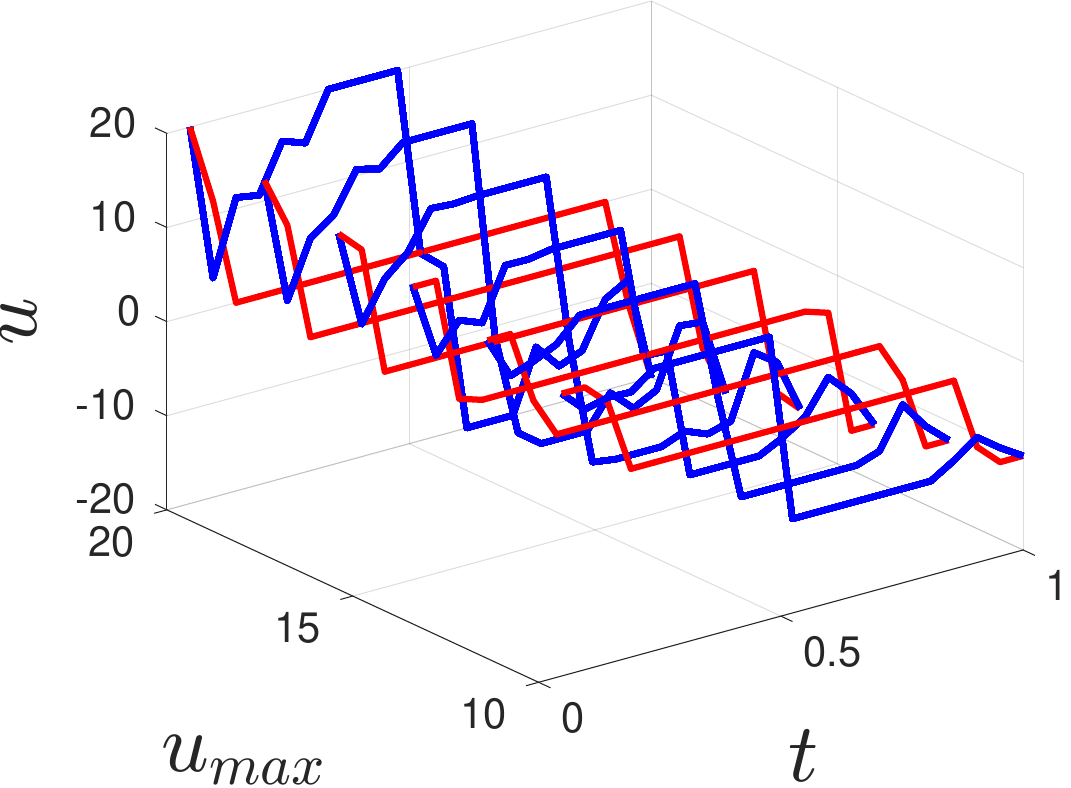}
            \end{minipage} 
            
            \begin{minipage}{0.31\textwidth}
                \vspace{0.2cm}
                \centering
                \includegraphics[width=\columnwidth]{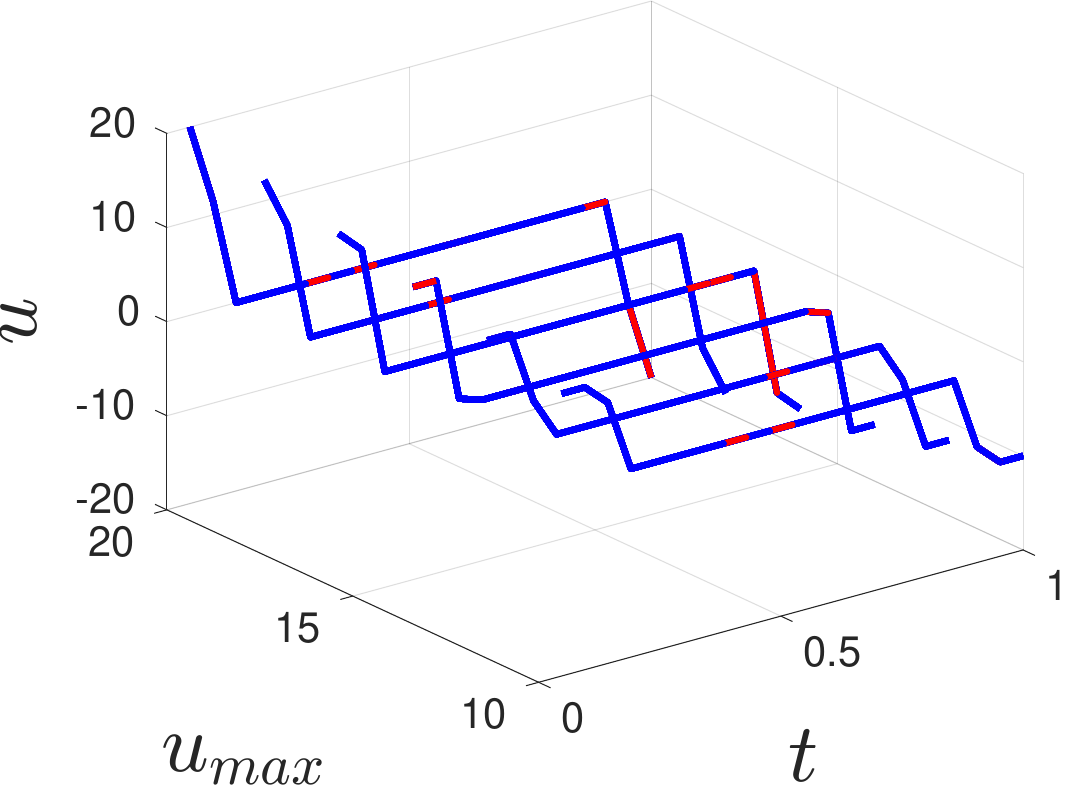}
            \end{minipage} 
        \end{tabular}
    \end{minipage}

    \begin{minipage}{\textwidth}
        \centering
        $N = 30$
        \vspace{0.2cm}
        \begin{tabular}{ccc}
            \begin{minipage}{0.31\textwidth}
                \vspace{0.2cm}
                \centering
                \includegraphics[width=\columnwidth]{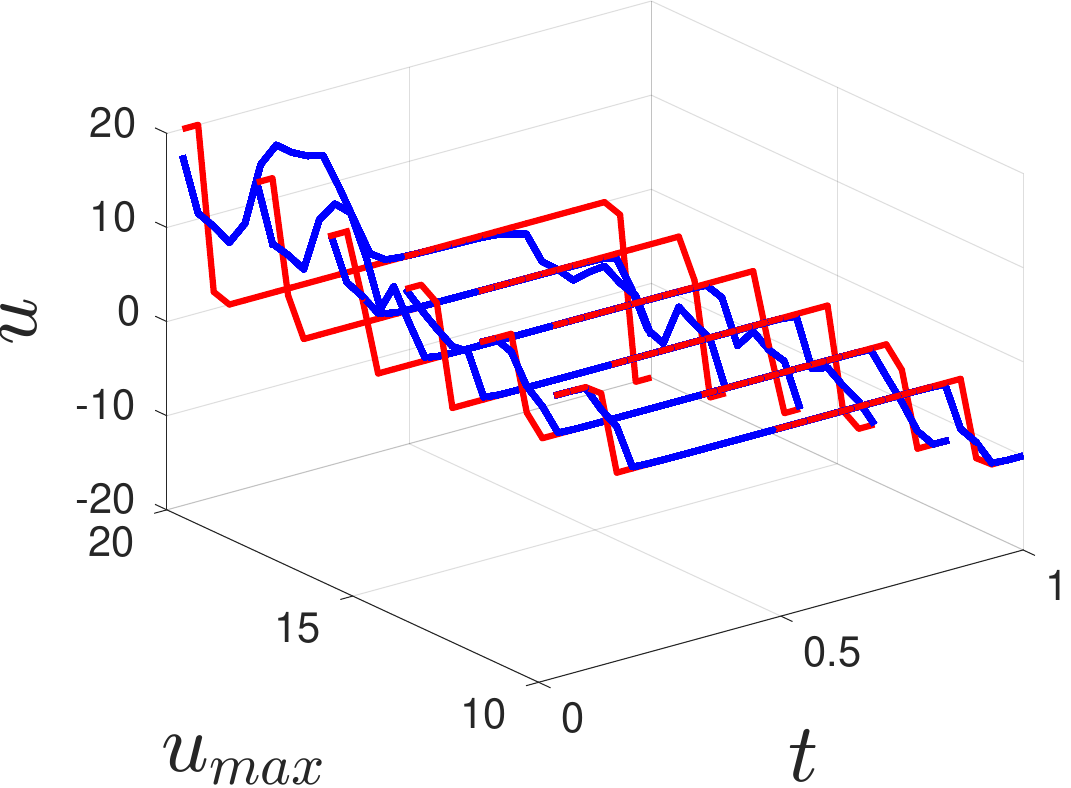}
                SSOS
            \end{minipage}
            
            \begin{minipage}{0.31\textwidth}
                \vspace{0.2cm}
                \centering
                \includegraphics[width=\columnwidth]{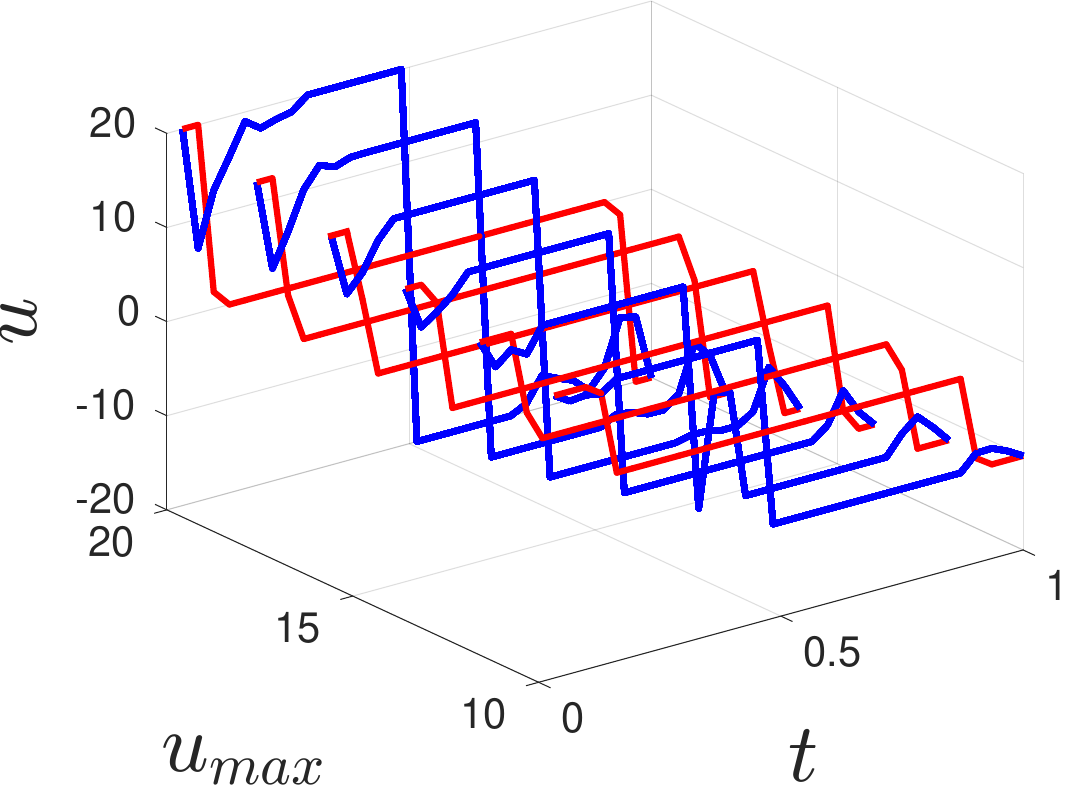}
                HSSOS1
            \end{minipage} 
            
            \begin{minipage}{0.31\textwidth}
                \vspace{0.2cm}
                \centering
                \includegraphics[width=\columnwidth]{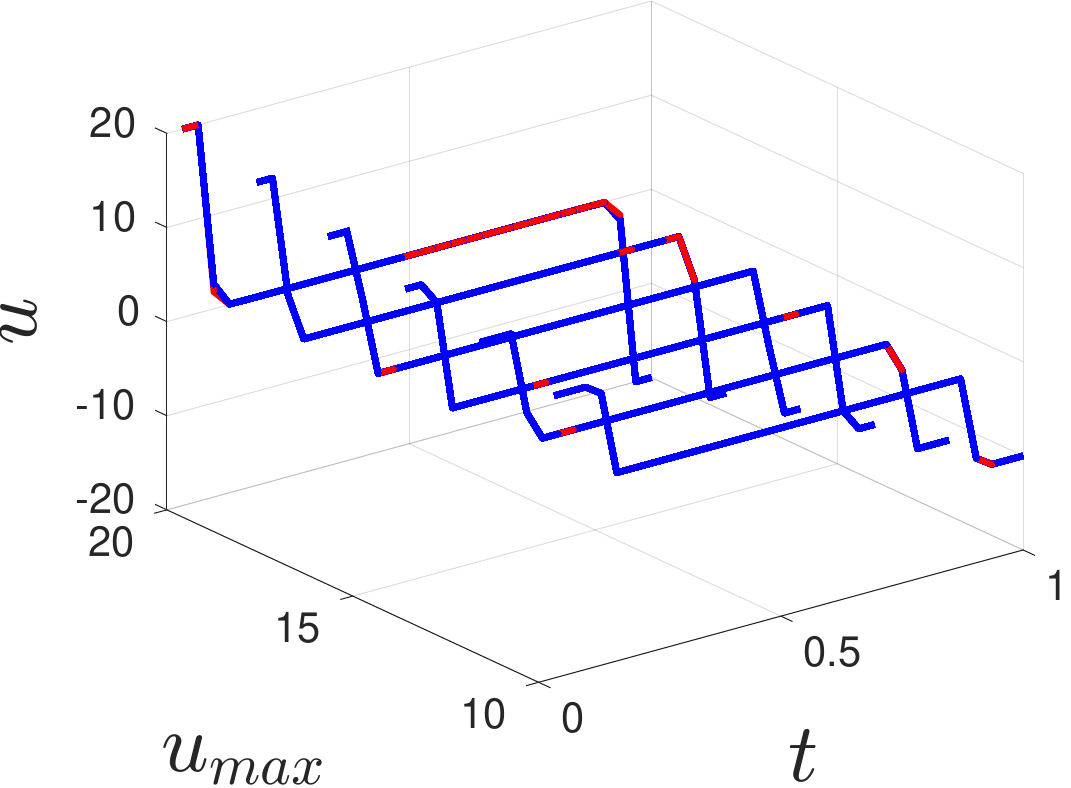}
                HSSOS3
            \end{minipage} 
        \end{tabular}
    \end{minipage}

    \caption{Comparison between SDP's solutions (blue lines) and solutions refined by fmincon (red lines) in the block-moving example. In HSSOS3, red lines and blue lines are nearly indistinguishable, indicating the attainment of tight solutions.}
    \label{fig:bmw}
\end{figure}

\subsection{Optimal control of Van der Pol}
Now we consider the optimal control problem for a Van der Pol oscillator~\cite{guckenheimer2003jads-vanderpol}, a highly nonlinear and potentially unstable system. Its continuous time dynamics is
\begin{equation}
    \label{eq:sec6-vanderpol-dyn}
    f(\bx, u) = \frac{d}{dt} \begin{bmatrix}
        x_1 \\ x_2
    \end{bmatrix} = \begin{bmatrix}
        (1 - x_2^2) x_1 - x_2 + u \\
        x_1
    \end{bmatrix}.
\end{equation}

Here $\bx = [x_1, x_2]^{\intercal}$ is the system state and $u$ is the control input. Utilizing the direct multiple shooting technique allows for the trajectory optimization problem as follows:
\begin{equation}
    \label{eq:sec6-traopt}
    \left\{ \begin{array}{rl}
    \min\limits_{\substack{
        u_k, \ k = 0, \dots, N-1 \\
        x_k, \ k = 0, \dots, N
    }} &\sum_{k=0}^{N-1} ( u_k^2 + \left\lVert x_k \right\rVert^2 ) \cdot h + \left\lVert x_N \right\rVert^2 \cdot dt \\
    \text{s.t.}\quad\quad\, & x_{k+1} = x_k + f(x_k, u_k) \cdot h, \quad k = 0, \dots, N-1, \\
    & u_{\max}^2 - u_k^2 \ge 0, \quad k = 0, \dots, N-1, \\
    & x_0 = x_{\text{init}},
    \end{array} \right.
\end{equation}
where $N$ is the total time steps and $h$ is the step length. Like~\eqref{eq:sec6-bmw-pop}, POP~\eqref{eq:sec6-traopt} also exhibits a chain-like csp by assigning the $N$ sequential cliques as $\{ (x_{k-1}, u_{k-1}, x_k) \}_{k=1}^N$. However,~\eqref{eq:sec6-traopt} does not fulfill the Archimedeanness assumption since the variables $\{x_k\}_{k=1}^N$ are not subject to any bound. With the relaxation order $k=2$, we incrementally vary $N$ from $10$ to $100$ in steps of $10$. At each $N$, the performance of three algorithms is assessed using $36$ predetermined initial states. Table~\ref{tab:v} presents the average results across these states. From the table, we can draw the conclusion that the extracted solutions of HSSOS3 are better than those of SSOS and HSSOS1 in terms of achieving one or two order of magnitudes lower sub-optimality gap $\eta$. Further comparison for solutions extracted from SDP relaxations and refined by fmincon are shown in Figure~\ref{fig:v}. Interestingly, despite the varying initial guesses supplied by the three algorithms, they all converge to identical refined solutions.


\begin{table}[htbp]
    \label{tab:v}
    \caption{Results of the Van der Pol example}
    \centering
    \renewcommand\arraystretch{1.2}
    \resizebox{\linewidth}{!}{
    \begin{tabular}{c|c|c|c|c|c|c|c|c|c}
         \multirow{2}{*}{$N$} & \multicolumn{3}{c|}{SSOS} & \multicolumn{3}{c|}{HSSOS1} & \multicolumn{3}{c}{HSSOS3}  \\
         \cline{2-10}
         & opt & time & gap & opt & time & gap & opt & time & gap \\
         \hline 
         $10$ & $11.559$ & $0.17$ & $-5.93$ & $11.454$ & $0.48$ & $-2.94$ & $11.559$ & $0.51$ & $-7.55$ \\
         \hline 
         $20$ & $18.457$ & $0.33$ & $-5.05$ & $18.230$ & $0.97$ & $-2.57$  & $18.534$ & $1.32$ & $-7.15$ \\
         \hline 
         $30$ & $23.485$ & $0.55$ & $-3.98$ & $23.012$ & $2.50$ & $-1.93$ & $23.734$ & $5.11$ & $-6.28$ \\
         \hline 
         $40$ & $26.728$ & $0.78$ & $-2.06$ & $25.760$ & $3.89$ & $-1.54$ & $27.419$ & $9.68$ & $-5.99$ \\
         \hline 
         $50$ & $28.122$ & $1.64$ & $-1.57$ & $27.418$ & $12.6$ & $-1.44$  & $29.780$ & $21.1$ & $-5.62$ \\
         \hline 
         $60$ & $28.655$ & $2.07$ & $-1.44$ & $28.434$ & $25.5$ & $-1.45$  & $31.058$ & $42.5$ & $-5.00$ \\
         \hline 
         $70$ & $28.782$ & $1.12$ & $-1.37$ & $29.118$ & $5.07$ & $-1.49$ & $31.768$ & $18.3$ & $-4.94$ \\
         \hline 
         $80$ & $28.874$ & $1.44$ & $-1.35$ & $29.582$ & $5.63$ & $-1.55$  & $32.131$ & $10.7$ & $-4.34$\\
         \hline 
         $90$ & $28.978$ & $1.52$ & $-1.34$ & $29.918$ & $9.79$ & $-1.65$  & $32.235$ & $33.6$ & $-4.22$\\
         \hline 
         $100$ & $29.033$ & $1.74$ & $-1.34$ & $30.198$ & $10.2$ & $-1.71$  & $32.257$ & $13.5$ & $-4.03$ \\
         
    \end{tabular}}
\end{table}


\begin{figure}[htbp]
    \begin{minipage}{\textwidth}
        \centering
        \begin{tabular}{cc}
            \begin{minipage}{0.47\textwidth}
                SSOS
                \vspace{0.2cm}
                \centering
                \includegraphics[width=\columnwidth]{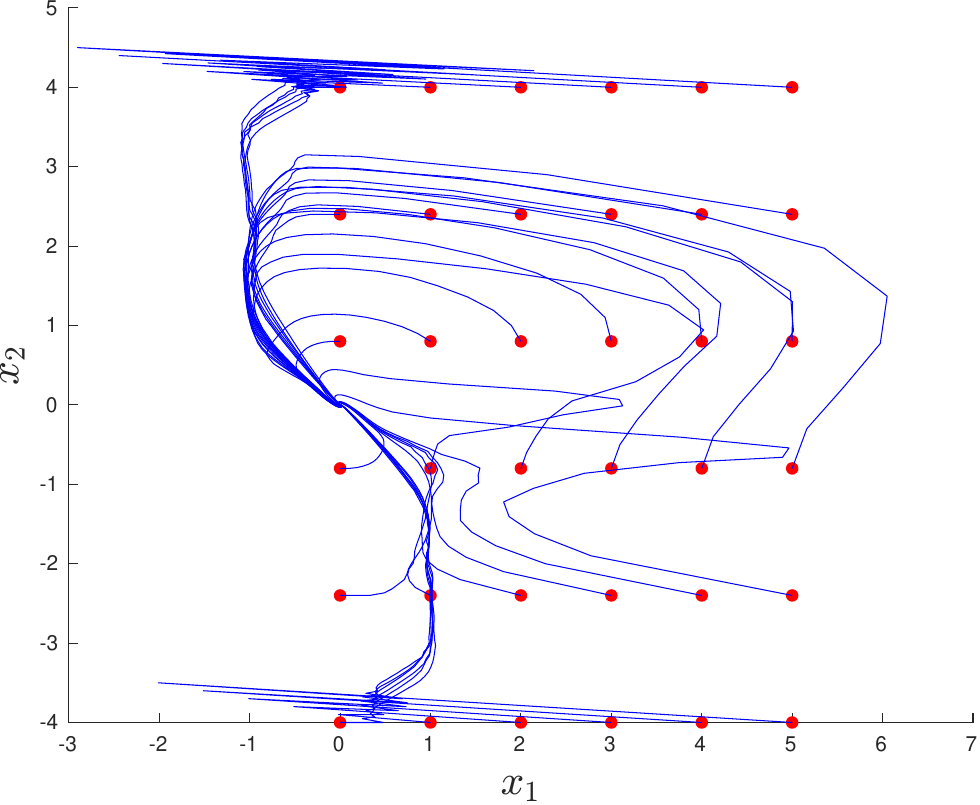}
            \end{minipage}
            
            \begin{minipage}{0.47\textwidth}
                HSSOS1
                \vspace{0.2cm}
                \centering
                \includegraphics[width=\columnwidth]{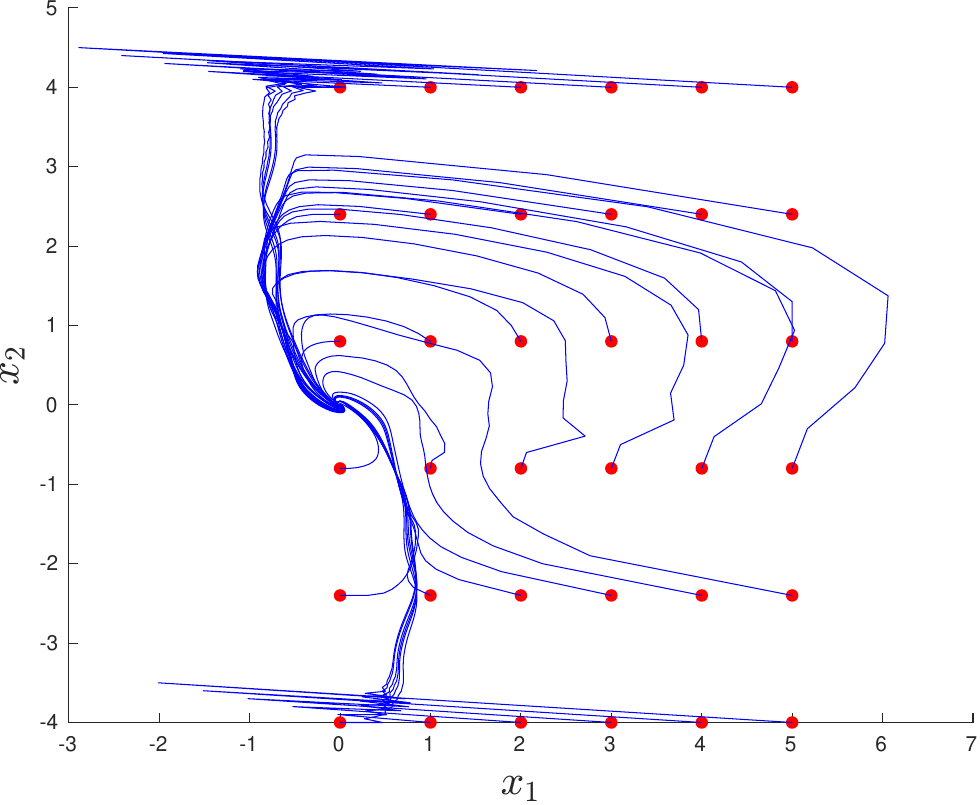}
            \end{minipage}  
        \end{tabular}
    \end{minipage}

    \begin{minipage}{\textwidth}
        \centering
        \vspace{0.2cm}
        \begin{tabular}{cc}
            \begin{minipage}{0.47\textwidth}
                HSSOS3
                \vspace{0.2cm}
                \centering
                \includegraphics[width=\columnwidth]{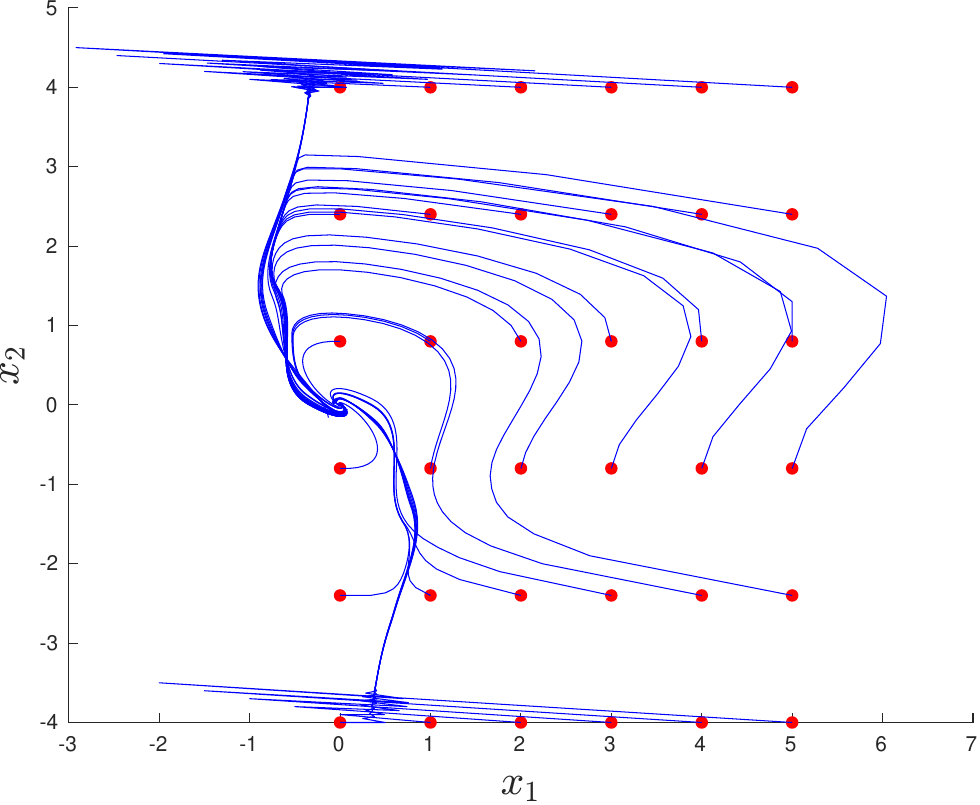}
            \end{minipage}
            
            \begin{minipage}{0.47\textwidth}
                fmincon
                \vspace{0.2cm}
                \centering
                \includegraphics[width=\columnwidth]{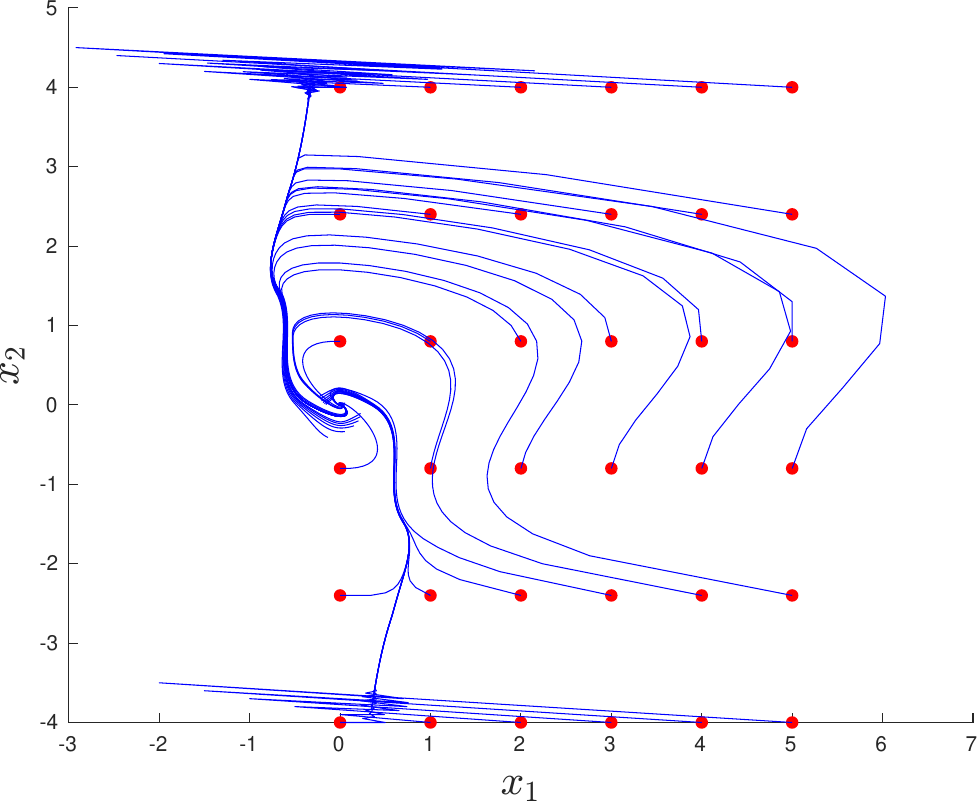}
            \end{minipage} 
        \end{tabular}
    \end{minipage}

    \caption{Comparison between SDP's solutions and solutions refined by fmincon in the Van der Pol example. Notably, all three algorithms' initial guesses lead to the same refined trajectories. Among these initial guesses, the one offered by HSSOS3 is of the best quality. }
    \label{fig:v}
\end{figure}

\section{Conclusions and discussions}
\label{sec:con}
In this paper, we propose the sparse homogenized Moment-SOS hierarchies to solve sparse polynomial optimization with unbonunded sets. We have shown the asymptotic convergence under the RIP and 
extensive numerical experiments demonstrate the power of our approach in solving problems with up to thousands of variables. Furthermore, we provide applications to two trajectory optimization problems and obtain global solutions of high accuracy.

Recently, polynomial upper bounds on the convergence rate of the Moment-SOS hierarchy with correlative sparsity \reff{spasos}--\reff{spamom} are obtained in 
\cite{converatesc}. It is promising to get similar convergence rates for our sparse homogenized hierarchies with additional considerations on the behaviour of $f$ at infinity of $K$.
In Section~\ref{sc:shmg2}, we propose 
two sparse homogenized Moment-SOS hierarchies without perturbations at the price of possibly increasing the maximal clique size. When tackling application problems, an interesting question is to explore how to construct correlative sparsity patterns with a small maximal clique size as the computational cost of sparse relaxations largely depends on this quantity.


\bibliographystyle{siamplain}
\bibliography{refer}

\begin{thebibliography}{10}

\bibitem{mosek}
{\sc E.~Andersen and K.~Andersen}, {\em {The Mosek interior point optimizer for linear programming: an implementation of the homogeneous algorithm}}, in High Performance Optimization, H.~Frenk, K.~Roos, T.~Terlaky, and S.~Zhang, eds., vol.~33 of Applied Optimization, Springer US, 2000, pp.~197--232.

\bibitem{aydinoglu2021tro-stabilization-complementary}
{\sc A.~Aydinoglu, P.~Sieg, V.~M. Preciado, and M.~Posa}, {\em Stabilization of complementarity systems via contact-aware controllers}, IEEE Transactions on Robotics, 38 (2021), pp.~1735--1754.

\bibitem{betts1998jgcd-traopt-survey}
{\sc J.~T. Betts}, {\em Survey of numerical methods for trajectory optimization}, Journal of guidance, control, and dynamics, 21 (1998), pp.~193--207.

\bibitem{curto2005truncated}
{\sc R.~Curto and L.~Fialkow}, {\em {Truncated K-moment problems in several variables}}, Journal of Operator Theory,  (2005), pp.~189--226.

\bibitem{grimm2007note}
{\sc D.~Grimm, T.~Netzer, and M.~Schweighofer}, {\em A note on the representation of positive polynomials with structured sparsity}, Archiv der Mathematik, 89 (2007), pp.~399--403.

\bibitem{guckenheimer2003jads-vanderpol}
{\sc J.~Guckenheimer, K.~Hoffman, and W.~Weckesser}, {\em The forced van der pol equation i: The slow flow and its bifurcations}, SIAM Journal on Applied Dynamical Systems, 2 (2003), pp.~1--35.

\bibitem{ha2009solving}
{\sc H.~V. Ha and T.~S. Pham}, {\em Solving polynomial optimization problems via the truncated tangency variety and sums of squares}, Journal of Pure and Applied Algebra, 213 (2009), pp.~2167--2176.

\bibitem{henrion2005detecting}
{\sc D.~Henrion and J.-B. Lasserre}, {\em Detecting global optimality and extracting solutions in gloptipoly}, in Positive Polynomials in Control, Springer, 2005, pp.~293--310.

\bibitem{huang2023finite}
{\sc L.~Huang, J.~Nie, and Y.-X. Yuan}, {\em Finite convergence of {M}oment-{SOS} relaxations with non-real radical ideals}, arXiv preprint arXiv:2309.15398,  (2023).

\bibitem{huang2023gmp}
{\sc L.~Huang, J.~Nie, and Y.-X. Yuan}, {\em Generalized truncated moment problems with unbounded sets}, Journal of Scientific Computing, 95 (2023), p.~15.

\bibitem{huang2023homogenization}
{\sc L.~Huang, J.~Nie, and Y.-X. Yuan}, {\em Homogenization for polynomial optimization with unbounded sets}, Mathematical Programming, 200 (2023), pp.~105--145.

\bibitem{jeyakumar2014polynomial}
{\sc V.~Jeyakumar, J.~Lasserre, and G.~Li}, {\em On polynomial optimization over non-compact semi-algebraic sets}, Journal of Optimization Theory and Applications, 163 (2014), pp.~707--718.

\bibitem{josz2018lasserre}
{\sc C.~Josz and D.~K. Molzahn}, {\em Lasserre hierarchy for large scale polynomial optimization in real and complex variables}, SIAM Journal on Optimization, 28 (2018), pp.~1017--1048.

\bibitem{kelly2017siam-intro-directcollocation}
{\sc M.~Kelly}, {\em An introduction to trajectory optimization: How to do your own direct collocation}, SIAM Rev., 59 (2017), pp.~849--904, \url{https://api.semanticscholar.org/CorpusID:33541404}.

\bibitem{klep2021sparse}
{\sc I.~Klep, V.~Magron, and J.~Povh}, {\em {Sparse Noncommutative Polynomial Optimization}}, Mathematical Programming,  (2021), pp.~1--41.

\bibitem{kojima2009note}
{\sc M.~Kojima and M.~Muramatsu}, {\em A note on sparse {SOS} and {SDP} relaxations for polynomial optimization problems over symmetric cones}, Computational Optimization and Applications, 42 (2009), pp.~31--41.

\bibitem{converatesc}
{\sc M.~Korda, V.~Magron, and R.~Zertuche}, {\em Convergence rates for sums-of-squares hierarchies with correlative sparsity}, Mathematical Programming,  (2024).

\bibitem{Las01}
{\sc J.~Lasserre}, {\em {Global optimization with polynomials and the problem of moments}}, SIAM Journal on Optimization, 11 (2001), pp.~796--817.

\bibitem{lasserre2006convergent}
{\sc J.~Lasserre}, {\em Convergent {SDP}-relaxations in polynomial optimization with sparsity}, SIAM Journal on optimization, 17 (2006), pp.~822--843.

\bibitem{lasserre2015introduction}
{\sc J.~Lasserre}, {\em {An Introduction to Polynomial and Semi-algebraic Optimization}}, vol.~52, Cambridge University Press, 2015.

\bibitem{laurent2007semidefinite}
{\sc M.~Laurent}, {\em Semidefinite representations for finite varieties}, Mathematical programming, 109 (2007), pp.~1--26.

\bibitem{laurent2009sums}
{\sc M.~Laurent}, {\em Sums of squares, moment matrices and optimization over polynomials}, Emerging applications of algebraic geometry,  (2009), pp.~157--270.

\bibitem{magron2018interval}
{\sc V.~Magron}, {\em Interval enclosures of upper bounds of roundoff errors using semidefinite programming}, ACM Transactions on Mathematical Software (TOMS), 44 (2018), pp.~1--18.

\bibitem{magron2021tssos}
{\sc V.~Magron and J.~Wang}, {\em {TSSOS}: a {J}ulia library to exploit sparsity for large-scale polynomial optimization}, 2021, \url{https://arxiv.org/abs/2103.00915}.

\bibitem{magron2023sparse}
{\sc V.~Magron and J.~Wang}, {\em Sparse polynomial optimization: theory and practice}, World Scientific, 2023.

\bibitem{pvrate}
{\sc N.~Mai and V.~Magron}, {\em {On the complexity of Putinar-Vasilescu's Positivstellensatz}}, Journal of Complexity,  (2022).
\newblock Accepted for publication.

\bibitem{mai2023sparse}
{\sc N.~Mai, V.~Magron, and J.~Lasserre}, {\em A sparse version of {Reznick’s} {Positivstellensatz}}, Mathematics of Operations Research, 48 (2023), pp.~812--833.

\bibitem{nonneg}
{\sc N.~H.~A. Mai, J.~Lasserre, and V.~Magron}, {\em Positivity certificates and polynomial optimization on non-compact semialgebraic sets}, Mathematical Programming,  (2021), pp.~1--43.

\bibitem{marshall2008positive}
{\sc M.~Marshall}, {\em {Positive Polynomials and Sums of Squares}}, no.~146, American Mathematical Soc., 2008.

\bibitem{newton2023sparse}
{\sc M.~Newton and A.~Papachristodoulou}, {\em Sparse polynomial optimisation for neural network verification}, Automatica, 157 (2023), p.~111233.

\bibitem{nie2014optimality}
{\sc J.~Nie}, {\em {Optimality conditions and finite convergence of {Lasserre’s} hierarchy}}, Mathematical programming, 146 (2014), pp.~97--121.

\bibitem{nie2019tight}
{\sc J.~Nie}, {\em Tight relaxations for polynomial optimization and {Lagrange} multiplier expressions}, Mathematical Programming, 178 (2019), pp.~1--37.

\bibitem{nie2023moment}
{\sc J.~Nie}, {\em {Moment and Polynomial Optimization}}, 2023.

\bibitem{pham2023tangencies}
{\sc T.~Pham}, {\em Tangencies and polynomial optimization}, Mathematical Programming, 199 (2023), pp.~1239--1272.

\bibitem{putinar1993positive}
{\sc M.~Putinar}, {\em Positive polynomials on compact semi-algebraic sets}, Indiana University Mathematics Journal, 42 (1993), pp.~969--984.

\bibitem{putinar1999positive}
{\sc M.~Putinar and F.~Vasilescu}, {\em Positive polynomials on semi-algebraic sets}, Comptes Rendus de l'Acad{\'e}mie des Sciences-Series I-Mathematics, 328 (1999), pp.~585--589.

\bibitem{putinar1999solving}
{\sc M.~Putinar and F.~Vasilescu}, {\em Solving moment problems by dimensional extension}, Comptes Rendus de l'Acad{\'e}mie des Sciences-Series I-Mathematics, 328 (1999), pp.~495--499.

\bibitem{qu2022correlative}
{\sc Z.~Qu and X.~Tang}, {\em A correlative sparse lagrange multiplier expression relaxation for polynomial optimization}, arXiv preprint arXiv:2208.03979,  (2022).

\bibitem{reznick2000some}
{\sc B.~Reznick}, {\em Some concrete aspects of {Hilbert's} 17th problem}, Contemporary Mathematics, 253 (2000), pp.~251--272.

\bibitem{schweighofer2006global}
{\sc M.~Schweighofer}, {\em Global optimization of polynomials using gradient tentacles and sums of squares}, SIAM Journal on Optimization, 17 (2006), pp.~920--942.

\bibitem{teng2023arxiv-geometricmotionplanning-liegroup}
{\sc S.~Teng, A.~Jasour, R.~Vasudevan, and M.~Ghaffari}, {\em Convex geometric motion planning on lie groups via moment relaxation}, 2023, \url{https://arxiv.org/abs/2305.13565}.

\bibitem{toh2018some}
{\sc K.~Toh}, {\em Some numerical issues in the development of {SDP} algorithms}, Informs OS Today, 8 (2018), pp.~7--20.

\bibitem{vui2008global}
{\sc H.~H. Vui and P.~T. Son}, {\em Global optimization of polynomials using the truncated tangency variety and sums of squares}, SIAM Journal on Optimization, 19 (2008), pp.~941--951.

\bibitem{waki2006sums}
{\sc H.~Waki, S.~Kim, M.~Kojima, and M.~Muramatsu}, {\em Sums of squares and semidefinite program relaxations for polynomial optimization problems with structured sparsity}, SIAM Journal on Optimization, 17 (2006), pp.~218--242.

\bibitem{nctssos}
{\sc J.~Wang and V.~Magron}, {\em Exploiting term sparsity in noncommutative polynomial optimization}, Computational Optimization and Applications, 80 (2021), pp.~483--521.

\bibitem{chordaltssos}
{\sc J.~Wang, V.~Magron, and J.~Lasserre}, {\em {Chordal-{TSSOS}: a moment-SOS hierarchy that exploits term sparsity with chordal extension}}, SIAM Journal on Optimization, 31 (2021), pp.~114--141.

\bibitem{tssos}
{\sc J.~Wang, V.~Magron, and J.~Lasserre}, {\em {{TSSOS}: a moment-{SOS} hierarchy that exploits term sparsity}}, SIAM Journal on optimization, 31 (2021), pp.~30--58.

\bibitem{cstssos}
{\sc J.~Wang, V.~Magron, J.~Lasserre, and N.~Mai}, {\em {{CS-TSSOS}: Correlative and term sparsity for large-scale polynomial optimization}}, ACM Trans. Math. Softw., 48 (2022), pp.~1--26.

\bibitem{acopf}
{\sc J.~Wang, V.~Magron, and J.~B. Lasserre}, {\em Certifying global optimality of {AC}-{OPF} solutions via sparse polynomial optimization}, Electric Power Systems Research, 213 (2022), p.~108683.

\bibitem{sparsedynsys}
{\sc J.~Wang, C.~Schlosser, M.~Korda, and V.~Magron}, {\em Exploiting term sparsity in moment-sos hierarchy for dynamical systems}, IEEE Transactions on Automatic Control,  (2023).

\end{thebibliography}

\end{document}